\documentclass[a4paper,reqno,10pt]{amsart}
\usepackage{eurosym}
\usepackage{amsfonts}
\usepackage{amsmath}
\usepackage{enumitem}
\usepackage{amssymb}
\usepackage{tikz-cd}
\usepackage{bbm}
\usepackage{amscd}
\usepackage[left=2cm,right=2cm,bottom=3cm,top=3cm]{geometry}
\usepackage{etoolbox}
\usepackage{filecontents}
\usepackage[backref=page,hidelinks]{hyperref}

\makeatletter
\patchcmd{\BR@backref}{\newblock}{\newblock(page~}{}{}
\patchcmd{\BR@backref}{\par}{)\par}{}{}
\makeatother
\newtheorem{theorem}{Theorem}[section]
\newtheorem*{theorem*}{Theorem}

\newtheorem{corollary}[theorem]{Corollary}

\newtheorem{definition}[theorem]{Definition}
\newtheorem{example}[theorem]{Example}

\newtheorem{lemma}[theorem]{Lemma}

\newtheorem{proposition}[theorem]{Proposition}
\newtheorem{remark}[theorem]{Remark}

\theoremstyle{definition}

\AtBeginDocument{   \def\MR#1{}
}
\DeclareMathOperator{\Mult}{Mult}

\begin{document}
\def\cprime{$'$}

\title[Boundary representations and compact rectangular matrix convex sets]{%
Boundary representations of operator spaces, and compact rectangular matrix
convex sets}
\author{Adam H. Fuller}
\address{Department of Mathematics, Ohio University, Athens, OH 45701}
\email{fullera@ohio.edu}
\urladdr{https://sites.google.com/site/afullermath/}
\author{Michael Hartz}
\address{Department of Mathematics, Washington University in St Louis, One
Brookings Drive, St. Louis, MO 63130}
\email{mphartz@wustl.edu}
\author{Martino Lupini}
\address{Mathematics Department\\
California Institute of Technology\\
1200 E. California Blvd\\
MC 253-37\\
Pasadena, CA 91125}
\date{\today }
\email{lupini@caltech.edu}
\urladdr{http://www.lupini.org/}
\thanks{M.H. was partially supported by an Ontario Trillium Scholarship and
a Feodor Lynen Fellowship. M.L. was partially supported by the NSF Grant
DMS-1600186. This work was initiated during a visit of M.H. at the
California Institute of Technology in the Spring 2016, and continued during
a visit of M.H. and M.L. at the Oberwolfach Mathematics Institute supported
by an Oberwolfach Leibnitz Fellowship. The authors gratefully acknowledge
the hospitality and the financial support of both institutions. }
\dedicatory{}

\begin{abstract}
We initiate the study of matrix convexity for operator spaces. We define the
notion of compact rectangular matrix convex set, and prove the natural
analogs of the Krein-Milman and the bipolar theorems in this context. We
deduce a canonical correspondence between compact rectangular matrix convex
sets and operator spaces. We also introduce the notion of boundary
representation for an operator space, and prove the natural analog of
Arveson's conjecture: every operator space is completely normed by its
boundary representations. This yields a canonical construction of the triple
envelope of an operator space.

%We introduce the notion of compact rectangular matrix convex set, which is
%the natural nonselfadjoint analog of the notion of compact rectangular
%matrix convex set. We then prove the Krein-Milman and bipolar theorem for
%compact rectangular matrix convex sets. We deduce that compact rectangular
%matrix convex sets are in natural one-to-one correspondence with operator
%%spaces, and that any operator space is completely normed by its
%rectangular-extreme matrix-valued completely bounded maps. We also introduce
%the notion of boundary representation for operator spaces, and prove that
%any operator space is completely normed by its boundary representations.
%This yields a canonical construction of the triple envelope of an operator
%space.
\end{abstract}
\maketitle
%\urladdr{www.calstatela.edu/faculty/adam-h-fuller}

%\thanks{}

\section{Introduction}

It is hard to overstate the importance of the theory of convexity in
analysis. This is all the more true in the study of operator systems, which
can be seen as the noncommutative analog of compact convex sets. Indeed,
given any operator system $S$, the space of matrix-valued unital completely
positive maps on $S$ is endowed with a natural notion of convex combinations
with matrix coefficients (matrix convex combination), and a topology which
is compact as long as one restricts the target to a fixed matrix algebra.
The \emph{compact matrix convex sets }that arise in this way have been
initially studied by Effros and Wittstock in \cite%
{wittstock_matrix_1984,effros_aspects_1978}. The program of developing the
theory of compact matrix convex sets as the noncommutative analog of compact
convex sets has been proposed by Effros in \cite{effros_aspects_1978}. This
program has been pursued in \cite%
{effros_matrix_1997,webster_krein-milman_1999}, where matricial
generalizations of the classical Krein-Milman and bipolar theorems are
proved. Compact matrix convex sets and the corresponding notion of matrix
extreme points have been subsequently studied in a number of papers. This
line of research has recently found outstanding applications. These include
the matrix convexity proof of Arveson's conjecture on boundary
representations due to Davidson and Kennedy \cite{davidson_choquet_2013}
building on previous work of Farenick \cite%
{farenick_extremal_2000,farenick_pure_2004}, and the work of Helton, Klep,
and McCullogh in free real algebraic geometry \cite%
{helton_every_2012,helton_free_2013}.

The main goal of this paper is to provide the nonselfadjoint analog of the
results above in the setting of operator spaces. Precisely, we introduce the
notion of compact \emph{rectangular }matrix convex set, which is the natural
analog of the notion of compact matrix convex set where convex combinations
with \emph{rectangular }matrices are considered. We then prove
generalizations of the Krein-Milman and bipolar theorems in the setting of
compact rectangular matrix convex sets. We then deduce that compact
rectangular matrix convex sets are in canonical functorial one-to-one
correspondence with operator spaces. It follows from this that any operator
space is completely normed by the matrix-valued completely contractive maps
that are rectangular matrix extreme points.

We also introduce the notion of boundary representation for operator spaces.
The natural operator space analog of Arveson's conjecture is then
established: any operator space is completely normed by its boundary
representations. This gives an explicit description of the triple envelope
of an operator space in terms of boundary representations. We also obtain in
this setting an analog of Arveson's boundary theorem. As an application, we
compute boundary representations for multiplier spaces associated with pairs
of reproducing kernel Hilbert spaces.

The results of this paper can be seen as the beginning of a convexity theory
approach to the study of operator spaces. Convexity theory has played a
crucial role in the setting of Banach spaces, such as in the groundbreaking
work of Alfsen and Effros on M-ideals in\ Banach spaces \cite%
{alfsen_structure_1972-1,alfsen_structure_1972-2} or the work of Lazar and
Lindenstrauss on $L^{1}$-predual spaces \cite%
{lazar_banach_1966,lazar_banach_1971}. This work can be seen as a first step
towards establishing noncommutative analogs of the results of Alfsen-Effros
and Lazar-Lindenstrauss mentioned above. We will see below that the crucial
notion of collinearity for bounded linear functionals on a Banach space,
which is of key importance for the work of Alfsen and Effros on facial cones
and M-ideals, has a natural interpretation in the setting of rectangular
matrix convexity.

The results of the present paper have already found application in \cite%
{lupini_uniqueness_2016}. The fact established here that an operator space
is completely normed by its matrix-valued completely contractive maps that
are rectangular matrix extreme is used there to prove that the
noncommutative Gurarij space introduced by Oikhberg in \cite%
{oikhberg_non-commutative_2006} is the unique separable nuclear operator
space with the property that the canonical map from the maximal TRO to the
triple envelope is injective.

This paper is divided into three sections, besides the introduction. In
Section \ref{Sec:boundary} we introduce the notion of boundary
representation for operator spaces, and prove that any operator space is
completely normed by its boundary representations. The boundary theorem for
operator spaces and applications to multiplier spaces for pairs of
reproducing kernel Hilbert spaces are also considered in this section. In
Section \ref{Sec:rectangular} we introduce the notion of compact rectangular
matrix convex set and rectangular matrix extreme point. We prove the
Krein-Milman and bipolar theorem for compact rectangular matrix convex sets,
and deduce the correspondence between compact rectangular matrix convex sets
and operator spaces. Finally in Section \ref{Section:gauge} we consider the
notion of matrix-gauged space. Such a concept has recently been introduced
by Russell in \cite{russell_characterizations_2015} in order to provide an
abstract characterization of selfadjoint subspaces of C*-algebras
(selfadjoint operator spaces), and to capture the injectivity property of $%
B(H)$ in such a category. We prove in Section \ref{Section:gauge} that the
construction of the injective envelope and the C*-envelope of an operator
system can be naturally generalized to the setting of matrix-gauged spaces.
A geometric approach to the study of matrix-gauged spaces and selfadjoint
operator spaces is also possible, as we show that matrix-gauged spaces are
in functorial one-to-one correspondence with compact matrix convex sets with
a distinguished matrix extreme point.

\section{Boundary representations and the Shilov boundary of an operator
space\label{Sec:boundary}}

\subsection{Notation and preliminaries}

Recall that a ternary ring of operators (TRO) $T$ is a subspace of the
C*-algebra $B(H)$ of bounded linear operators on a Hilbert space $H$ that is
closed under the triple product $\left( x,y,z\right) \mapsto xy^{\ast }z$.
An important example of a TRO is the space $B(H,K)$, where $H,K$ are Hilbert spaces.
A TRO has a canonical operator space structure coming from the inclusion $%
T\subset B(H)$, which does not depend on the concrete representation of $T$
as a ternary ring of operators on $H$. A triple morphism between TROs is a
linear map that preserves the triple product. Any TRO can be seen as the 1-2
corner of a canonical C*-algebra $\mathcal{L}\left( T\right) $ called its 
\emph{linking algebra}. A triple morphism between TROs can be seen as the
1-2 corner of a *-homomorphism between the corresponding linking algebras 
\cite{hamana_triple_1999}; see also \cite[Corollary 8.3.5]%
{blecher_operator_2004}.

The notions of (nondegenerate, irreducible, faithful) representations admit
natural generalizations from\ C*-algebras to TROs. A representation of a TRO 
$T$ is a triple morphism $\theta :T\rightarrow B(H,K)$ for some Hilbert
spaces $H,K$. A linear map $\psi: T \to B(H,K)$ is \emph{nondegenerate }if, whenever $%
p,q $ are projections in $B(H)$ and $B(K)$, respectively, such that $q\theta
(x)=\theta (x)p=0$ for every $x\in T$, one has $p=0$ and $q=0$. Similarly a
representation $\theta $ of $T$ is \emph{irreducible} if , whenever $p,q$
are projections in $B(H)$ and $B(K)$, respectively, such that $q\theta
(x)p+\left( 1-q\right) \theta (x)\left( 1-p\right) =\theta (x)$ for every $%
x\in T$ (equivalently, $q\theta (x)=\theta (x)p$ for every $x\in T$), one
has $p=1$ and $q=1$, or $p=0$ and $q=0$. Finally, $\theta $ is called \emph{%
faithful }if it is injective or, equivalently, completely isometric. Various
characterizations of nondegenerate and irreducible representations are
obtained in \cite[Lemma 3.1.4 and Lemma 3.1.5]{bohle_k-theory_2011}. A
concrete TRO $T\subset B(H,K)$ is said to act nondegenerately or irreducibly
if the corresponding inclusion representation is nondegenerate or
irreducible, respectively.

In the following we will use frequently without mention the
Haagerup-Paulsen-Wittstock extension theorem \cite[Theorem 8.2]%
{paulsen_completely_2002}, asserting that, if $H,K$ are Hilbert spaces, then
the space $B(H,K)$ of bounded linear operators from $H$ to $K$ is injective
in the category of operator spaces and completely contractive maps.\ We will
also often use the canonical way, due to Paulsen, to assign to an operator
space $X\subset B(H,K)$ an operator system $\mathcal{S}(X)\subset B(K\oplus
H)$. This operator system, called the \emph{Paulsen system}, is defined to
be the space of operators%
\begin{equation*}
\left\{ 
\begin{bmatrix}
\lambda I_{K} & x \\ 
y^{\ast } & \mu I_{H}%
\end{bmatrix}%
:x,y\in X,\lambda ,\mu \in \mathbb{C}\right\}
\end{equation*}%
where $I_{H}$ and $I_{K}$ denote the identity operator on $H$ and $K$,
respectively. Any completely contractive map $\phi :X\rightarrow Y$ between
operator spaces extends canonically to a unital completely positive map $%
\mathcal{S}(\phi ):\mathcal{S}(X)\rightarrow \mathcal{S}(Y)$ defined by%
\begin{equation*}
\begin{bmatrix}
\lambda I_{K} & x \\ 
y^{\ast } & \mu I_{H}%
\end{bmatrix}%
\mapsto 
\begin{bmatrix}
\lambda I_{K} & \phi (x) \\ 
\phi (y)^{\ast } & \mu I_{H}%
\end{bmatrix}%
\text{,}
\end{equation*}%
see \cite[Lemma 8.1]{paulsen_completely_2002}. (The Paulsen system is
defined in \cite[Chapter 8]{paulsen_completely_2002} and \cite[Section 1.3]%
{blecher_operator_2004} only in the case when $H=K$. The same proofs from 
\cite[Chapter 8]{paulsen_completely_2002} and \cite[Section 1.3]%
{blecher_operator_2004} apply with no change to this more general situation.)

\subsection{Dilations of rectangular operator states\label{Subs:dilations}}

Suppose that $X$ is an operator space. A \emph{rectangular operator state }%
on $X$ is a nondegenerate linear map $\phi :X\rightarrow B(H,K)$ such that $%
\left\Vert \phi \right\Vert _{cb}=1$. We say that a rectangular operator
state $\psi :X\rightarrow B(\widetilde{H},\widetilde{K})$ is a \emph{%
dilation }of $\phi $ if there exist linear isometries $v:H\rightarrow 
\widetilde{H}$ and $w:K\rightarrow \widetilde{K}$ such that $w^{\ast }\psi
(x)v=\phi (x)$ for every $x\in X$. The same proof as \cite[Theorem 8.4]%
{paulsen_completely_2002} gives the following.

\begin{proposition}
\label{Proposition:dilation-TRO}Any rectangular operator state $\phi
:T\rightarrow B(H,K)$ on a TRO $T\subset B\left( H_{0},K_{0}\right) $ can be
dilated to a nondegenerate triple morphism $\theta :T\rightarrow B(%
\widetilde{H},\widetilde{K})$. If $H_{0},K_{0},H,K$ are finite-dimensional,
then one can take $\widetilde{H}$ and $\widetilde{K}$ to be
finite-dimensional.
\end{proposition}

In order to prove Proposition \ref{Proposition:dilation-TRO} one can proceed
as in \cite[Theorem 8.4]{paulsen_completely_2002}, by replacing $M_{2}\left(
A\right) $ for a given C*-algebra $A$ with the C*-algebra generated by $%
\mathcal{S}\left( T\right) $ inside $B(K_{0}\oplus H_{0})$. It is clear that
in Proposition \ref{Proposition:dilation-TRO} one can choose $\theta $ and
the linear isometries $v:H\rightarrow \widetilde{H}$ and $w:K\rightarrow 
\widetilde{K}$ in such a way that $\widetilde{K}$ is the linear span of $%
\theta(T) \theta(T)^* wK\cup \theta \left( T\right) vH$, and $%
\widetilde{H}$ is the linear span of $\theta(T)^* \theta(T) vH\cup
\theta \left( T\right) ^{\ast }wK$. In this case, we call such a dilation $%
\theta $ a \emph{minimal} dilation of $\phi $.
In the sequel, it will often be convenient to identify $H$ with a subspace
of $\widetilde H$ and $K$ with a subspace of $\widetilde K$.

\begin{definition}
  Let $\phi: X \to B(H,K)$ be a rectangular operator state and
  let $\psi :X\rightarrow B(\widetilde H,\widetilde K)$ be a dilation of $\phi$.
  We can assume that $H \subset \widetilde H$ and $K \subset \widetilde K$.
  Let $p$ be the orthogonal projection from $\widetilde H$ onto $H$
  and let $q$ be the orthogonal projection from $\widetilde K$ onto $K$.
  The dilation $\psi $ is \emph{trivial} if%
\begin{equation*}
\psi \left( x\right) =q\psi \left( x\right) p+ (1 - q) \psi \left( x\right) (1- p)
\end{equation*}%
for every $x\in X$. The operator state $\phi $ on an operator space $X$ is 
\emph{maximal }if it has no nontrivial dilation.
\end{definition}

It is clear that, when $X$ is an operator system, $\widetilde H= \widetilde K$, $q=p$, and $\psi $
is a unital completely positive map, the notion of trivial dilation as above
recovers the usual notion of trivial dilation.

\begin{definition}
Suppose that $X$ is a subspace of a TRO $T$ such that $T$ is generated as a
TRO by $X$. The operator state $\phi $ on $X$ has the\emph{\ unique
extension property} if any rectangular operator state $\widetilde{\phi }$ of 
$T$ whose restriction to $X$ coincides with $\phi $ is automatically a
triple morphism.
\end{definition}

We now observe, that a rectangular operator state on an operator space is
maximal if and only if it has the unique extension property. The analogous
fact for unital completely positive maps on operator systems is well known;
see \cite{arveson_noncommutative_2008}.

\begin{lemma}
\label{lem:positive_cone} For a TRO $T$, the set of positive elements of the 
$C^{\ast }$-algebra $TT^{\ast }$ is the closed convex cone generated by $%
\{xx^{\ast }:x\in T\}$.
\end{lemma}

\begin{proof}
Let $C\subset TT^{\ast }$ denote the closed convex cone generated by $%
\{xx^{\ast }:x\in T\}$. It is clear that $C$ is contained in the set of
positive elements of $TT^{\ast }$. Conversely, suppose that $a\in TT^{\ast }$
is positive. By the remarks at the beginning of Section 2.2 of \cite%
{eleftherakis_strong_2017}, the $C^{\ast }$-algebra $TT^{\ast }$ admits a
contractive approximate identity $(e_{i})$ of elements of the form%
\begin{equation*}
\sum_{j}x_{j}x_{j}^{\ast }
\end{equation*}%
where $x_{j}\in T$. For such an element of $TT^{\ast }$ one has that%
\begin{equation*}
a^{1/2}\left( \sum_{j}x_{j}x_{j}^{\ast }\right)
a^{1/2}=\sum_{j}(a^{1/2}x_{j})(a^{1/2}x_{j})^{\ast }\in C\text{,}
\end{equation*}
as $T$ is a left $T T^*$-module.
Thus, $a^{1/2}e_{i}a^{1/2}\in C$ for every $i$. It follows that $%
a=\lim_{i}a^{1/2}e_{i}a^{1/2}\in C$.
\end{proof}

\begin{lemma}\label{Lemma:trivial}
Let $T$ be a TRO, let $\psi: T \to B(\widetilde H,\widetilde K)$ be a completely contractive
linear map and suppose that $H \subset \widetilde H$ and $K \subset \widetilde K$
are closed subspaces with corresponding orthogonal projections $p \in B(\widetilde H)$ and $q \in B(\widetilde K)$.
If the map
\begin{equation*}
\theta: T \to B(H, K), \quad x \mapsto q \psi(x) p,
\end{equation*}
is a non-degenerate triple morphism, then $\psi$ is a trivial dilation of $\theta$.
\end{lemma}

\begin{proof}
By dilating $\psi $ if necessary, we may assume without loss of generality
that $\psi $ is a triple morphism. Thus, there exists a unique $\ast $%
-homomorphism 
\begin{equation*}
\sigma :TT^{\ast }\rightarrow B(\widetilde K)\quad \text{ such that }\quad \sigma
(xy^{\ast })=\psi (x)\psi (y)^{\ast }\quad (x,y\in T).
\end{equation*}%
The assumption further implies that there exists a $\ast $-homomorphism 
\begin{equation*}
\pi :TT^{\ast }\rightarrow B(K)\quad \text{ such that}\quad \pi (xy^{\ast
})=q\psi (x)p\psi (y)^{\ast }q.
\end{equation*}%
Since $\theta $ is non-degenerate, $\pi $ is non-degenerate as well.
Consider now the map 
\begin{equation*}
\varphi :TT^{\ast }\rightarrow B(K),\quad a\mapsto q\sigma (a)q-\pi (a).
\end{equation*}%
We claim that $\varphi =0$. To this end, observe that if $x\in T$, then 
\begin{equation*}
\varphi (xx^{\ast })=q\psi (x)(1-p)\psi (x)^{\ast }q\geq 0,
\end{equation*}%
so $\varphi $ is a positive map by Lemma \ref{lem:positive_cone}. Let $%
(e_{i})$ be an approximate identity for $TT^{\ast }$. Since $\pi $ is
non-degenerate, $\pi (e_{i})$ tends to $q$ in the strong operator topology.
Since $q\sigma (e_{i})q\leq q$, it follows that $\varphi (e_{i})=q\sigma
(e_{i})q-\pi (e_{i})$ tends to zero in the strong operator topology.
Combining this with positivity of $\varphi $, it follows that $\varphi =0$
(see, e.g. \cite[Lemma I.9.5]{davidson_c*-algebras_1996}).

In particular, we see that for $x\in T$, 
\begin{equation*}
0=\varphi (xx^{\ast })=q\psi (x)(1-p)\psi (x)^{\ast }q=[q\psi
(x)(1-p)][q\psi (x)(1-p)]^{\ast },
\end{equation*}%
so that $q\psi (x)(1-p)=0$. A similar argument, replacing $TT^{\ast }$ with $%
T^{\ast }T$, shows that $p\psi (x)^{\ast }(1-q)=0$. Thus, $\psi $ is a
trivial dilation of $\theta $.
\end{proof}

\begin{proposition}
\label{Proposition:maximal-uep}Suppose that $\phi :X\rightarrow B(H,K)$ is a
rectangular operator state of $X$, and $T$ is a TRO containing $X$ as a
generating subspace. Then $\phi $ is maximal if and only if it has the
unique extension property.
\end{proposition}

\begin{proof}
Suppose initially that $\phi $ is maximal. Let $\widetilde{\phi }%
:T\rightarrow B(H,K)$ be an extension of $\phi $. Let $\theta :T\rightarrow
B(\widetilde{H},\widetilde{K})$ be a dilation of $\widetilde{\phi }$ to a
triple morphism. We can identify $H$ with a subspace of $\widetilde{H}$ and $%
K$ with a subspace of $\widetilde{K}$. Let $p$ and $q$ be the orthogonal
projections of $\widetilde{H}$ and $\widetilde{K}$ onto $H$ and $K$,
respectively. We have that $\widetilde{\phi }(x)=q\theta (x)|_{H}$ for every 
$x\in T$. The restriction of $\theta $ to $X$ is a rectangular operator
state that dilates $\phi $. By maximality of $\phi $, we can conclude that $%
\theta (x)=q\theta (x)p+\left( 1-q\right) \theta (x)\left( 1-p\right) $ for
every $x\in X$. Since $X$ generates $T$ as a TRO and $\theta $ is a triple
morphism, it follows that this identity holds for every $x\in T$. It follows
that $\widetilde{\phi }$ is a triple morphism as well.

Suppose now that $\phi $ has the unique extension property. Let $\psi
:X\rightarrow B(\widetilde{H},\widetilde{K})$ be a dilation of $\phi $. As
above, we will identify $H$ and $K$ as subspaces of $\widetilde{H}$ and $%
\widetilde{K}$, respectively, and denote by $p$ and $q$ the corresponding
orthogonal projections. We can extend $\psi $ to a rectangular operator
state $\psi :T\rightarrow B(\widetilde{H},\widetilde{K})$. Observe that $%
x\mapsto q\psi (x)|_{H}$ is a rectangular operator state extending $\phi $.
Since $\phi $ has the unique extension property, $x\mapsto q\psi (x)|_{H}$
is a triple morphism. Hence from Lemma \ref{Lemma:trivial} we
can conclude that $\psi $ is a trivial dilation of $\phi $. Since $\psi $
was arbitrary, we can conclude that $\phi $ is maximal.
\end{proof}

Simple examples show that the implication ``unique extension property implies maximal''
of Proposititon \ref{Proposition:maximal-uep} may fail if $\phi$ is a degenerate completely
contractive map. Indeed, there are degenerate representations of TROs which have non-trivial dilations.

\subsection{Boundary representations\label{Subs:boudnary}}

Suppose that $X$ is an operator space, and $T$ is a TRO containing $X$ as a
generating subspace.

\begin{definition}
\label{Definition:bound-rep}A \emph{boundary representation }for $X$ is a
rectangular operator state $\phi :X\rightarrow B(H,K)$ with the property
that any rectangular operator state on $T$ extending $X$ is an irreducible
representation of $T$.
\end{definition}

In other words, a rectangular operator state $\phi: X \to B(H,K)$ is a
boundary representation for $X$ if and only if it has the unique extension
property, and the unique extension of $\phi$ to $T$ is an irreducible
representation of $T$. In the following we will identify a boundary
representation of $X$ with its unique extension to an irreducible
representation of $T$. It follows from Proposition \ref%
{Proposition:maximal-uep} that the notion of boundary representation does
not depend on the concrete realization of $X$ as a space of operators. We
remark that our terminology differs slightly from Arveson's original use of
the term boundary representation in the context of operator systems \cite%
{arveson_subalgebras_1969}. Indeed, for Arveson, a boundary representation
is a representation of the $C^*$-algebra generated by the operator system.
More precisely, if $S$ is an operator system that generates the $C^*$%
-algebra $A$, then according to Arveson, a boundary representation for $S$
is an irreducible representation $\pi$ of $A$ such that $\pi$ is the unique
completely positive extension of $\pi|_S$. We follow the convention, which
is for example used in \cite{davidson_choquet_2013}, that a boundary
representation of $S$ is a unital completely positive map $\phi:S \to B(H)$
such that every extension of $\phi$ to a completely positive map on $A$ is
an irreducible representation of $A$. Since this notion does not depend on
the concrete representation of $S$, these two points of view are equivalent.

In the rest of this section, we will observe that the boundary
representations of $X$ completely norm $X$. This will be deduced from the
corresponding fact about operator systems, proved in \cite%
{arveson_noncommutative_2008} in the separable case and in \cite%
{davidson_choquet_2013} in full generality.

\begin{proposition}
\label{Proposition:Paulsen-boundary}Suppose $\omega :\mathcal{S}%
(X)\rightarrow B(L_{\omega })$ is a boundary representation of the Paulsen
system $\mathcal{S}(X)$ associated with $X$. Then one can decompose $%
L_{\omega }$ as an orthogonal direct sum $K_{\omega }\oplus H_{\omega }$ in
such a way that $\omega =\mathcal{S}(\psi )$ for some boundary
representation $\psi :X\rightarrow B(H_{\omega },K_{\omega })$ of $X$.
\end{proposition}

\begin{proof}
Suppose that $T\subset B(H,K)$ is a TRO containing $X$ as a generating
subspace. Let $A$ be the C*-algebra generated by $\mathcal{S}(X)$ inside $%
B(K\oplus H)$. Observe that%
\begin{equation*}
A=\left\{ 
\begin{bmatrix}
x_{11}+\lambda I_{K} & x_{12} \\ 
x_{21} & x_{22}+\mu I_{H}%
\end{bmatrix}%
:x_{11}\in TT^{\ast }\text{, }x_{12}\in T\text{, }x_{21}\in T^{\ast }\text{, 
}x_{22}\in T^{\ast }T\text{, }\lambda ,\mu \in \mathbb{C}\right\} \text{.}
\end{equation*}%
Since $\omega $ is a boundary representation of $\mathcal{S}(X)$, it extends
to an irreducible representation $\omega :A\rightarrow B(L_{\omega })$. Set%
\begin{equation*}
q_{\omega }:=\omega \left( 
\begin{bmatrix}
I_{K} & 0 \\ 
0 & 0%
\end{bmatrix}%
\right) \text{, and }p_{\omega }:=\omega \left( 
\begin{bmatrix}
0 & 0 \\ 
0 & I_{H}%
\end{bmatrix}%
\right) \text{.}
\end{equation*}%
Observe that $p_{\omega },q_{\omega }\in B(L_{\omega })$ are orthogonal
projections such that $p_{\omega }+q_{\omega }=I_{L_{\omega }}$. Denote by $%
K_{\omega }$ the range of $q_{\omega }$ and by $H_{\omega }$ the range of $%
p_{\omega }$. The fact that $\omega $ is a unital *-homomorphism implies
that, with respect to the decomposition $L_{\omega }=K_{\omega }\oplus
H_{\omega }$ one has that%
\begin{equation*}
\omega =%
\begin{bmatrix}
\sigma & \theta \\ 
\theta ^{\ast } & \pi%
\end{bmatrix}%
\end{equation*}%
where $\theta :T\rightarrow B(H_{\omega },K_{\omega })$ is a triple morphism.

We claim that $\theta $ is irreducible. Suppose that $p\in B(H_{\omega })$
and $q\in B(K_{\omega })$ are projections such that $q\theta (x)=\theta (x)p$
for every $x\in T$. Since $\sigma \left( ab^{\ast }\right) =\theta (a)\theta
(b)^{\ast }$ and $\pi \left( a^{\ast }b\right) =\theta (a)^{\ast }\theta (b)$
for every $a,b\in T$, we can conclude that $\left( q\oplus p\right) \omega
(x) = \omega(x) (q \oplus p)$ for every $x\in A$. Since $\omega $ is an
irreducible representation of $A$, it follows that $q \oplus p = 1$ or $q
\oplus p = 0$. This concludes the proof that $\theta $ is irreducible.

Denote by $\psi $ the restriction of $\theta $ to $X$. Observe that $\omega =%
\mathcal{S}(\psi )$. We claim that $\psi $ is maximal. Indeed, let $\phi
:X\rightarrow B(\widetilde{H},\widetilde{K})$ be a dilation of $\psi $. We
can identify $H$ and $K$ as subspaces of $\widetilde{H}$ and $\widetilde{K}$%
, with corresponding orthogonal projections $p$ and $q$. Then $\mathcal{S}%
(\phi ):\mathcal{S}(X)\rightarrow B(\widetilde{K}\oplus \widetilde{H})$ is a
dilation of $\omega $. By maximality of $\omega $, we have that%
\begin{equation*}
\begin{bmatrix}
q & 0 \\ 
0 & p%
\end{bmatrix}%
\omega (x)=\omega (x)%
\begin{bmatrix}
q & 0 \\ 
0 & p%
\end{bmatrix}%
\end{equation*}%
for every $x\in \mathcal{S}(X)$. It follows that $q\phi (x)=\phi (x)p$ for
every $x\in X$. This shows that $\phi $ is a trivial dilation of $\psi $,
concluding the proof that $\psi $ is maximal.
\end{proof}

The following result is now an immediate consequence of Proposition \ref%
{Proposition:Paulsen-boundary} and \cite[Theorem 3.4]{davidson_choquet_2013}.

\begin{theorem}
\label{Theorem:shilov-boundary}Suppose that $X$ is an operator space. Then $%
X $ is completely normed by its boundary representations.
\end{theorem}

\subsection{Rectangular extreme points and pure unital completely positive
maps\label{Subs:rectangular}}

Suppose that $X$ is an operator space, and $\phi :X\rightarrow B(H,K)$ is a
completely contractive linear map. A\emph{\ rectangular operator convex
combination} is an expression $\phi =\alpha _{1}^{\ast }\phi _{1}\beta
_{1}+\cdots +\alpha _{n}^{\ast }\phi _{n}\beta _{n}$, where $\beta
_{i}:H\rightarrow H_{i}$ and $\alpha _{i}:K\rightarrow K_{i}$ are linear
maps, and $\phi _{i}:X\rightarrow B(H_{i},K_{i})$ are completely contractive
linear maps for $i=1,2,\ldots ,\ell $ such that $\alpha _{1}^{\ast }\alpha
_{1}+\cdots +\alpha _{n}^{\ast }\alpha _{n}=1$, and $\beta _{1}^{\ast }\beta
_{1}+\cdots +\beta _{n}^{\ast }\beta _{n}=1$. Such a rectangular convex
combination is \emph{proper }if $\alpha _{i},\beta _{i}$ are surjective, and 
$\emph{trivial}$ if $\alpha _{i}^{\ast }\alpha _{i}=\lambda _{i}1$, $\beta
_{i}^{\ast }\beta _{i}=\lambda _{i}1$, and $\alpha _{i}^{\ast }\phi
_{i}\beta _{i}=\lambda _{i}\phi $ for some $\lambda _{i}\in \left[ 0,1\right]
$.

\begin{definition}
\label{Definition:rectangular-operator-extreme}A completely contractive map $%
\phi :X\rightarrow B(H,K)$ is a \emph{rectangular operator extreme point} if
any proper rectangular operator convex combination $\phi =\alpha _{1}^{\ast
}\phi _{1}\beta _{1}+\cdots +\alpha _{n}^{\ast }\phi _{n}\beta _{n}$ is
trivial.
\end{definition}

Suppose now that $X$ is an operator system. An \emph{operator state }on $X$
is a unital completely positive map $\phi :X\rightarrow B(H)$. An \emph{%
operator convex combination }is an expression $\phi =\alpha _{1}^{\ast }\phi
_{1}\alpha _{1}+\cdots +\alpha _{n}^{\ast }\phi _{n}\alpha _{n}$, where $%
\alpha _{i}:H\rightarrow H_{i}$ are linear maps, and $\phi _{i}:X\rightarrow
B(H_{i})$ are operator states for $i=1,2,\ldots ,\ell $ such that $\alpha
_{1}^{\ast }\alpha _{1}+\cdots +\alpha _{n}^{\ast }\alpha _{n}=1$. Such an
operator convex combination is \emph{proper }if $\alpha _{i}$ is right
invertible for $i=1,2,\ldots ,\ell $, and \emph{trivial} if $\alpha
_{i}^{\ast }\alpha _{i}=\lambda _{i}1$ and $\alpha _{i}^{\ast }\phi
_{i}\alpha _{i}=\lambda _{i}\phi $ for some $\lambda _{i}\in \left[ 0,1%
\right] $.

We say that $\phi $ is an \emph{operator extreme point }if any proper
operator convex combination $\phi =\alpha _{1}^{\ast }\phi _{1}\alpha
_{1}+\cdots +\alpha _{n}^{\ast }\phi _{n}\alpha _{n}$ is trivial. The proof
of \cite[Theorem B]{farenick_extremal_2000} shows that an operator state is
an operator extreme point if and only if it is a pure element in the cone of
completely positive maps.

When $H$ is finite-dimensional, the notion of proper operator convex
combination coincides with the notion of proper matrix convex combination
from \cite{webster_krein-milman_1999}. In this case, the notion of operator
extreme point coincides with the notion of matrix extreme point from \cite[%
Definition 2.1]{webster_krein-milman_1999}.

\begin{lemma}
\label{Lemma:below-Paulsen}Suppose that $\phi :X\rightarrow B(H,K)$ is a
completely contractive linear map, $\Psi :\mathcal{S}(X)\rightarrow
B(K\oplus H)$ is a completely positive map such that $\mathcal{S}(\phi
)-\Psi $ is completely positive. Suppose that $\Psi (1)$ is an invertible
element of $B(K\oplus H)$. Then there exist positive invertible elements $%
a\in B(K)$ and $b\in B(H)$, and a completely contractive map $\psi
:X\rightarrow B(H,K)$ such that%
\begin{equation*}
\Psi \left( 
\begin{bmatrix}
\lambda & x \\ 
y^{\ast } & \mu%
\end{bmatrix}%
\right) =%
\begin{bmatrix}
a & 0 \\ 
0 & b%
\end{bmatrix}%
\begin{bmatrix}
\lambda I_{K} & \psi (x) \\ 
\psi (y)^{\ast } & \mu I_{H}%
\end{bmatrix}%
\begin{bmatrix}
a & 0 \\ 
0 & b%
\end{bmatrix}%
\text{.}
\end{equation*}
\end{lemma}

\begin{proof}
Fix a concrete representation $X\subset B(L_{0},L_{1})$ of $X$. In this case 
$\mathcal{S}(X)\subset B(L_{1}\oplus L_{0})$. Set $\Phi :=\mathcal{S}(\phi )$
and let $T\subset B(L_{0},L_{1})$ denote the TRO generated by $X$. By
Arveson's extension theorem, we may extend $\Phi $ and $\Psi $ to the
C*-algebra $A$ generated by $\mathcal{S}(X)$ inside $B(L_{1}\oplus L_{0})$
in such a way that $\Phi - \Psi$ is still completely positive.
In the following we regard $\Phi ,\Psi $ as maps from $A$ to $B(K\oplus H)$.
Since $\Phi -\Psi $ is completely positive, the argument in the proof of 
\cite[Theorem 8.3]{paulsen_completely_2002} shows that there exist linear
maps $\varphi _{1}:TT^{\ast }+\mathbb{C}I_{L_{1}}\rightarrow B(K)$ and $%
\varphi _{0}:T^{\ast }T+\mathbb{C}I_{L_{0}}\rightarrow B(H)$ such that 
\begin{equation*}
\Psi \left( 
\begin{bmatrix}
x & 0 \\ 
0 & 0%
\end{bmatrix}%
\right) =%
\begin{bmatrix}
\varphi _{1}(x) & 0 \\ 
0 & 0%
\end{bmatrix}%
.
\end{equation*}%
and 
\begin{equation*}
\Psi \left( 
\begin{bmatrix}
0 & 0 \\ 
0 & y%
\end{bmatrix}%
\right) =%
\begin{bmatrix}
0 & 0 \\ 
0 & \varphi _{0}(y)%
\end{bmatrix}%
.
\end{equation*}%
In particular, $w=\Psi (1)$ is a diagonal element, and the unital completely
positive map $\Psi _{0}=w^{-1/2}\Psi w^{-1/2}$ satisfies 
\begin{equation*}
\Psi _{0}\left( 
\begin{bmatrix}
I_{L_{1}} & 0 \\ 
0 & 0%
\end{bmatrix}%
\right) =%
\begin{bmatrix}
I_{K} & 0 \\ 
0 & 0%
\end{bmatrix}%
\quad \text{ and }\quad \Psi _{0}\left( 
\begin{bmatrix}
0 & 0 \\ 
0 & I_{L_{0}}%
\end{bmatrix}%
\right) =%
\begin{bmatrix}
0 & 0 \\ 
0 & I_{H}%
\end{bmatrix}%
.
\end{equation*}%
It follows that the two projections $%
\begin{bmatrix}
I_{L_{1}} & 0 \\ 
0 & 0%
\end{bmatrix}%
$ and $%
\begin{bmatrix}
0 & 0 \\ 
0 & I_{L_{0}}%
\end{bmatrix}%
$ belong to the multiplicative domain of $\Psi _{0}$ \cite[Proposition 1.3.11%
]{blecher_operator_2004}, so that there exists a completely contractive map $%
\psi :X\rightarrow B(H,K)$ such that 
\begin{equation*}
\Psi _{0}\left( 
\begin{bmatrix}
\lambda I_{L_{1}} & x \\ 
y^{\ast } & \mu I_{L_{0}}%
\end{bmatrix}%
\right) =%
\begin{bmatrix}
\lambda I_{K} & \psi (x) \\ 
\psi (y)^{\ast } & \mu I_{H}%
\end{bmatrix}%
,
\end{equation*}%
which finishes the proof.
\end{proof}

\begin{proposition}
\label{Proposition:characterize-rectangular-extreme}Suppose that $\phi
:X\rightarrow B(H,K)$ is a completely contractive map and $\mathcal{S}(\phi
):\mathcal{S}(X)\rightarrow B(K\oplus H)$ is the associated unital
completely positive map defined on the Paulsen system. The following
assertions are equivalent:

\begin{enumerate}
\item $\mathcal{S}(\phi )$ is a pure completely positive map;

\item $\mathcal{S}(\phi )$ is an operator extreme point;

\item $\phi $ is a rectangular operator extreme point.
\end{enumerate}
\end{proposition}

\begin{proof}
We have already observed that the equivalence of (1) and (2) holds, as the
argument in the proof of \cite[Theorem B]{farenick_extremal_2000} shows.

(2)$\implies $(3) Suppose that $\phi =\alpha _{1}^{\ast }\phi _{1}\beta
_{1}+\cdots +\alpha _{\ell }^{\ast }\phi _{\ell }\beta _{\ell }$ is a proper
rectangular matrix convex combination. Define $\gamma _{i}=\alpha _{i}\oplus
\beta _{i}$ for $i=1,2,\ldots ,\ell $. Then we have that $\mathcal{S}(\phi
)=\gamma _{1}^{\ast }\mathcal{S}\left( \phi _{1}\right) \gamma _{1}+\cdots
+\gamma _{\ell }^{\ast }\mathcal{S}\left( \phi _{\ell }\right) \gamma _{\ell
}$ is a proper matrix convex combination. Since by assumption $\mathcal{S}%
(\phi )$ is an operator extreme point in the state space of $\mathcal{S}(X)$,
we can conclude that the proper matrix convex combination $\gamma _{1}^{\ast
}\mathcal{S}\left( \phi _{1}\right) \gamma _{1}+\cdots +\gamma _{\ell
}^{\ast }\mathcal{S}\left( \phi _{\ell }\right) \gamma _{\ell }$ is trivial.
This implies that the proper rectangular matrix convex combination $\alpha
_{1}^{\ast }\phi _{1}\beta _{1}+\cdots +\alpha _{\ell }^{\ast }\phi _{\ell
}\beta _{\ell }$ is trivial as well.

(3)$\implies $(1) Suppose that $\mathcal{S}(\phi )=\Psi _{1}+\Psi _{2}$ for
some completely positive maps $\Psi _{1},\Psi _{2}:\mathcal{S}(X)\rightarrow
B(K\oplus H)$. Fix $\varepsilon >0$ and define $\Xi _{i}=\left(
1-\varepsilon \right) \Psi _{i}+\left( \varepsilon /2\right) \mathcal{S}%
(\phi )$ for $i=1,2$. Then $\Xi _{1},\Xi _{2}:\mathcal{S}(X)\rightarrow
B(K\oplus H)$ are completely positive maps such that $\Xi _{1}+\Xi _{2}=%
\mathcal{S}(\phi )$ and $\Xi _{i}(1)$ is invertible for $i=1,2$; cf.\ the
proof of \cite[Lemma 2.3]{davidson_choquet_2013}. By Lemma \ref%
{Lemma:below-Paulsen} we have that, for $i=1,2$, 
\begin{equation*}
\Xi _{i}\left( 
\begin{bmatrix}
\lambda & x \\ 
y^{\ast } & \mu%
\end{bmatrix}%
\right) =%
\begin{bmatrix}
a_{i} & 0 \\ 
0 & b_{i}%
\end{bmatrix}%
\begin{bmatrix}
\lambda 1 & \psi _{i}(x) \\ 
\psi _{i}(y)^{\ast } & \mu 1%
\end{bmatrix}%
\begin{bmatrix}
a_{i} & 0 \\ 
0 & b_{i}%
\end{bmatrix}%
\end{equation*}%
for some positive invertible elements $a_{i}\in B(K)$, $b_{i}\in B(H)$, and
completely contractive $\psi _{i}:X\rightarrow B(H,K)$. Thus we have
that $\phi =a_{1}\psi _{1}b_{1}+a_{2}\psi _{2}b_{2}$ is a proper rectangular
operator convex combination. By assumption, we have that $a_{i}^{2}=t_{i}1$, 
$b_{i}^{2}=t_{i}1$, and $a_{i}\psi _{i}b_{i}=t_{i}\phi $ for some $t_{i}\in %
\left[ 0,1\right] $ and $i=1,2$. It follows that $\Xi _{i}=t_{i}\mathcal{S}%
(\phi )$ for $i=1,2$. Since this is true for every $\varepsilon $, it
follows that the $\Psi _{i}$ are also scalar multiples of $\mathcal{S}(\phi
) $. This concludes the proof that $\mathcal{S}(\phi )$ is pure.
\end{proof}

The following corollary is an immediate consequence of Proposition \ref%
{Proposition:characterize-rectangular-extreme},\ Proposition \ref%
{Proposition:Paulsen-boundary}, and \cite[Theorem 2.4]{davidson_choquet_2013}%
.

\begin{corollary}
Suppose that $\phi :X\rightarrow B(H,K)$ is a rectangular operator state. If 
$\phi $ is rectangular operator extreme, then $\phi $ admits a dilation to a
boundary representation of $X$.
\end{corollary}

Suppose that $X$ is a Banach space.
We regard $X$ as an operator space endowed with its canonical \emph{minimal }%
operator space structure obtained from the canonical inclusion of $X$ in the
C*-algebra $C\left( \mathrm{\mathrm{Ball}}\left( X^{\prime }\right) \right) $,
where $X'$ denotes the dual space of $X$.
In \cite{alfsen_structure_1972-1} Alfsen
and Effros considered the following notion. Suppose that $\phi _{0},\phi
_{1} $ are nonzero contractive linear functionals on $X$. Then $\phi
_{0},\phi _{1}$ are \emph{codirectional }if $\left\Vert \phi _{0}+\phi
_{1}\right\Vert =\left\Vert \phi _{0}\right\Vert +\left\Vert \phi
_{1}\right\Vert $. This is equivalent to the assertion that
$||\phi_0 + \phi_1|| \mathcal{S}(%
\frac{\phi _{0}+\phi _{1}}{\left\Vert \phi _{0}+\phi _{1}\right\Vert }%
)=\left\Vert \phi _{0}\right\Vert \mathcal{S}(\frac{\phi _{0}}{\left\Vert
\phi _{0}\right\Vert })+\left\Vert \phi _{1}\right\Vert \mathcal{S(}\frac{%
\phi _{1}}{\left\Vert \phi _{1}\right\Vert })$. The relation $\prec $ on
nonzero contractive linear functionals is defined by setting $\phi \prec
\psi $ if and only if either $\phi =\psi $ or $\phi $ and $\psi -\phi $ are
codirectional. This is equivalent to the assertion that $\left\Vert \phi
\right\Vert \mathcal{S}(\frac{\phi }{\left\Vert \phi \right\Vert })\leq 
\left\Vert \psi
\right\Vert \mathcal{S}(\frac{\psi }{\left\Vert \psi \right\Vert })$.

\begin{proposition}
\label{Proposition:pointy}Suppose that $X$ is a Banach space and $\phi $ is
a bounded linear functional on $X$ of norm $1$. The following statements are
equivalent:

\begin{enumerate}
\item $\phi $ is an extreme point of $\mathrm{Ball}(X^{\prime })$;

\item if $\psi \in \mathrm{Ball}(X^{\prime })$ is a nonzero linear
functional such that $\psi \prec \phi $, then $\psi $ is a scalar multiple
of $\phi $;

\item $\phi $ is a rectangular extreme point.
\end{enumerate}
\end{proposition}

\begin{proof}
The implications (3)$\Rightarrow $(1)$\Rightarrow $(2) are straightforward.

(2)$\Rightarrow $(3) Suppose that, for every nonzero $\psi \in \mathrm{Ball}%
(X^{\prime })$, $\psi \prec \phi $ implies that $\psi $ is a scalar multiple
of $\phi $. The proof of Proposition \ref%
{Proposition:characterize-rectangular-extreme} shows that it suffices to
show that every proper rectangular convex combination of two elements is
trivial. Thus, consider a rectangular convex combination $\phi =s_{0}%
\overline{t}_{0}\phi _{0}+s_{1}\overline{t}_{1}\phi _{1}$ for some $\phi
_{0},\phi _{1}\in \mathrm{Ball}(X^{\prime })$ and non-zero $s_{0},
s_{1},t_{0},t_{1}\in \mathbb{C} $ such that $\left\vert s_{0}\right\vert
^{2}+\left\vert s_{1}\right\vert ^{2}=1$ and $\left\vert t_{0}\right\vert
^{2}+\left\vert t_{1}\right\vert ^{2}=1$. Observe that 
\begin{equation*}
1=\left\Vert s_{0}\overline{t}_{0}\phi _{0}+s_{1}\overline{t}_{1}\phi
_{1}\right\Vert \leq \left\Vert s_{0}\overline{t}_{0}\phi _{0}\right\Vert
+\left\Vert s_{1}\overline{t}_{1}\phi _{1}\right\Vert \leq \left\vert s_{0}%
\overline{t}_{0}\right\vert \left\Vert \phi _{0}\right\Vert +\left\vert s_{1}%
\overline{t}_{1}\right\vert \left\Vert \phi _{1}\right\Vert \leq \left\vert
s_{0}\overline{t}_{0}\right\vert +\left\vert s_{1}\overline{t}%
_{1}\right\vert \leq 1\text{.}
\end{equation*}%
Hence $\left\Vert s_{0}\overline{t}_{0}\phi _{0}\right\Vert +\left\Vert s_{1}%
\overline{t}_{1}\phi _{1}\right\Vert =\left\Vert \phi _{0}\right\Vert
=\left\Vert \phi _{1}\right\Vert =1$ and $|s_0| = |t_0|$ and $|s_1|= |t_1|$.
By hypothesis we have that $s_{0}%
\overline{t}_{0}\phi _{0}=\rho _{0}\phi $ and $s_{1}\overline{t}_{1}\phi
_{1}=\rho _{1}\phi $ for some $\rho _{0},\rho _{1}\in \mathbb{C}$.
In particular, $|\rho_i| = |s_i|^2 = |t_i|^2$ for $i=0,1$.
Since $\phi =s_{0} \overline{t}_{0}\phi _{0}+s_{1}\overline{t}_{1}\phi _{1}$,
it follows that $\rho_0 + \rho_1 = 1$. Combined with $|\rho_0| + |\rho_1| = 1$,
this implies that $\rho_0,\rho_1 \in [0,1]$, so that the rectangular convex combination
was trivial.
\end{proof}

\subsection{TRO-extreme points}

Suppose that $X$ is an operator space and $\varphi :X\rightarrow B(H,K)$ is
a completely contractive map. We say that $\varphi $ is a TRO-extreme point
if whenever $\varphi =\alpha _{1}^{\ast }\varphi _{1}\beta _{1}+\cdots
+\alpha _{\ell }^{\ast }\varphi _{\ell }\beta _{\ell }$ is a proper
rectangular matrix convex combination such that $\varphi _{i}:X\rightarrow
B(H,K)$ for $i=1,2,\ldots ,\ell $, then $\alpha _{i}^{\ast }\alpha
_{i}=t_{i}1$, $\beta _{i}^{\ast }\beta _{i}=t_{i}1$, and $\alpha _{i}^{\ast
}\varphi _{i}\beta _{i}=t_{i}\varphi $ for some $t_{i}\in \left[ 0,1\right] $%
. When $H,K$ are finite-dimensional, this is equivalent to requiring that
there exist unitaries $u_{i}\in M_{n}\left( \mathbb{C}\right) $ and $%
w_{i}\in M_{m}\left( \mathbb{C}\right) $ such that $\varphi _{i}=u_{i}^{\ast
}\varphi w_{i}$. This can be seen arguing as in the proof of Lemma \ref%
{Lemma:rectangular-extreme} below. The notion of TRO-extreme point can be
seen as the operator space analog of the notion of C*-extreme point
considered in \cite%
{hopenwasser_c*-extreme_1981,farenick_c*-extreme_1993,farenick_c*-extreme_1997}%
.

A similar proof as \cite[Theorem B]{farenick_extremal_2000} gives the
following lemma.

\begin{lemma}
Let $X$ be an operator space, and $A$ be the C*-algebra generated by $%
\mathcal{S}(X)$. If $\varphi :X\rightarrow B(H,K)$ is a TRO-extreme point
such that the range of $\varphi $ is an irreducible subspace of $B(H,K)$,
then there exists a pure unital completely positive map $\Phi :A\rightarrow
B(K\oplus H)$ that extends $S(\varphi)$.
\end{lemma}

Using this lemma, one can prove similarly as \cite[Theorem C]%
{farenick_extremal_2000} the following fact:

\begin{proposition}
Suppose that $T$ is a TRO, and $\varphi :T\rightarrow B(H,K)$ is a
TRO-extreme point. Then there exist pairwise orthogonal projections $\left(
p_{i}\right) _{i\in I}$ in $B(H)$ and $\left( q_{i}\right) _{i\in I}$ in $%
B(K)$ such that $p_{i}\varphi q_{i}$ is rectangular operator extreme for
every $i$, and $p_{i}\varphi q_{j}=0$ for every $i\neq j$.
\end{proposition}

\subsection{The Shilov boundary of an operator space\label{Subs:shilov}}

Suppose that $X$ is an operator space. A \emph{triple cover }of $X$ is a
pair $\left( \iota ,T\right) $ where $T$ is a TRO and $\iota :X\rightarrow T$
is a completely isometric linear map whose range is a subspace of $T$ that
generates $T$ as a TRO. Among the triple covers there exists a canonical
one: the \emph{triple envelope}. This is the (unique) triple cover $(\iota
_{e},\mathcal{T}_{e}(X))$ with the property that for any other triple cover $%
\left( \iota ,T\right) $ of $X$, there exists a triple morphism $\theta
:T\rightarrow \mathcal{T}_{e}(X)$ such that $\theta \circ \iota =\iota _{e}$%
.\ The existence of the triple envelope was established by Hamana in \cite%
{hamana_injective_1992} using the construction of the injective envelope of
an operator space; see also \cite[Section 4.4]{blecher_operator_2004}.

The triple envelope of an operator space is referred to as the \emph{%
(noncommutative) Shilov boundary }in \cite{blecher_shilov_2001}. It is
remarked in \cite{blecher_shilov_2001} at the beginning of Section 4,
referring to the Shilov boundary of an operator space, that
\textquotedblleft the spaces above are not at the present time defined
canonically\textquotedblright , and \textquotedblleft this lack of
canonicity is always a potential source of blunders in this area, if one is
not careful about various identifications." We remark here that the theory
of boundary representations provides a canonical construction of the Shilov
boundary of an operator space $X$. Indeed one can consider $\iota
_{e}:X\rightarrow B(H,K)$ to be the direct sum of all the boundary
representations for $X$, and then let $\left( \iota _{e},\mathcal{T}%
_{e}(X)\right) $ be the subTRO of $B(H,K)$ generated by the image of $\iota
_{e}$. Proposition \ref{Proposition:maximal-uep} implies that $\iota_e$ is
maximal, so we may argue as in the proof of \cite[Theorem 4.1]%
{dritschel_boundary_2005} that $\mathcal{T}_{e}(X)$ is indeed the triple
envelope of $X $. It also follows that if $X$ has a completely isometric
boundary representation $\theta :X\rightarrow B(H)$, then the triple
envelope of $X$ is the TRO generated by the range of $\theta $ inside $B(H)$.

\subsection{The Shilov boundary of a Banach space\label{Subs:shilov-banach}}

A TRO $T$ is \emph{commutative }if $xy^{\ast }z=zy^{\ast }x$ for every $%
x,y,z\in T$. Several equivalent characterizations of commutative TROs are
provided in \cite[Proposition 8.6.5]{blecher_operator_2004}. Suppose that $E$
is a locally trivial line bundle over a locally compact Hausdorff space $U$.
Then the space $\Gamma _{0}(E)$ of continuous sections of $E$ that vanish at
infinity is a commutative TRO such that $\Gamma _{0}(E)^{\ast }\Gamma
_{0}(E)=C_{0}(E)$. Conversely, it is observed in \cite[Section 4]%
{blecher_shilov_2001}---see also \cite{dupre_banach_1983}---that any
commutative TRO is of this form. One can also describe the commutative TROs
as the $C_{\sigma }$-spaces from the Banach space literature.

Suppose that $E$ is a locally trivial line bundle over a locally compact
Hausdorff space $U$ with point at infinity $\infty $ and $X\subset \Gamma
_{0}(E)$ be a closed subspace. Assume that the set of elements $\left\{
\left\langle x,y\right\rangle :x,y\in X\right\} $ of $C_{0}\left( U\right) $
separates the points of $U$ and does not identitically vanish at any point
of $U$. This is equivalent to the assertion that $X$ generates $\Gamma
_{0}(E)$ as a TRO, as proved in \cite[Theorem 4.20]{blecher_shilov_2001}. An
irreducible representation of $\Gamma _{0}(E)$ is of the form $x\mapsto
x\left( \omega _{0}\right) $ for some $\omega _{0}\in U$. A linear map from $%
\Gamma _{0}(E)$ to $\mathbb{C}$ of norm $1$ has the form $x\mapsto \int
x\left( \omega \right) d\mu \left( \omega \right) $ for some Borel
probability measure $\mu $ on $U$. We say that $\mu $ is a \emph{%
representing measure} for $\omega _{0}\in U$ if $\int x\left( \omega \right)
d\mu \left( \omega \right) =x\left( \omega _{0}\right) $ for every $x\in X$.
A point $\omega _{0}\in U$ is a \emph{Choquet boundary point }if the point
mass at $\omega _{0}$ is the unique representing measure for $\omega _{0}$.
It follows from the observations above that $\omega _{0}$ is a Choquet
boundary point for $X$ if and only if the map $x\mapsto x\left( \omega
_{0}\right) $ is a boundary representation for $X$. The Choquet boundary $%
\mathrm{Ch}(X)$ of $X$ is the set of Choquet boundary points of $X$.

Suppose that $\partial _{S}X\cup \left\{ \infty \right\} $ is the closure of 
$\mathrm{Ch}(X)\cup \left\{ \infty \right\} $ inside $U\cup \left\{ \infty
\right\} $. Then it follows from Theorem \ref{Theorem:shilov-boundary} that $%
E|_{\partial _{S}X}$ is the Shilov boundary of $X$ in the sense of \cite[%
Theorem 4.25]{blecher_shilov_2001}. This means that the linear map $%
X\rightarrow \Gamma _{0}(E|_{\partial _{S}X})$, $x\mapsto x|_{\partial
_{S}X} $ is isometric, and for any locally trivial line bundle over a
locally compact Hausdorff space $V$ and linear isometry $J:X\rightarrow
\Gamma _{0}\left( V\right) $ with the property that the set $\left\{
J(x)^{\ast }J(y):x,y\in X\right\} $ separates the points of $V$ and does not
identically vanish at any point of $V$, there exists a proper continuous
injection $\varphi :\partial _{S}X\rightarrow V$ with the property that $%
J(x)\circ \varphi =x|_{\partial _{S}X}$ for every $x\in X$. This gives a
canonical construction of the Shilov boundary of a Banach space, analogous
to the canonical construction of a Shilov boundary of a unital function
space; see \cite[Section 4.1]{blecher_operator_2004}.

\subsection{The rectangular boundary theorem}

Arveson's boundary theorem \cite[Theorem 2.1.1]{Arveson72} asserts that if $%
S\subset B(H)$ is an operator system which acts irreducibly on $H$ such that
the C*-algebra $C^{\ast }(S)$ contains the algebra of compact operators $%
\mathcal{K}(H)$, then the identity representation of $C^{\ast }(S)$ is a
boundary representation for $S$ if and only if the quotient map $%
B(H)\rightarrow B(H)/\mathcal{K}(H)$ is not completely isometric on $S$.

The following result is a rectangular generalization of Arveson's boundary
theorem.

\begin{theorem}
\label{T:rectangular_boundary_thm} Let $X\subset B(H,K)$ be an operator
space such that the TRO $T$ generated by $X$ acts irreducibly and such that $%
T\cap \mathcal{K}(H,K)\neq \{0\}$. Then the identity representation of $T$
is a boundary representation for $X$ if and only if the quotient map $%
B(H,K)\rightarrow B(H,K)/\mathcal{K}(H,K)$ is not completely isometric on $X$%
.
\end{theorem}

\begin{proof}
Suppose first that the quotient map $\pi $ is completely isometric on $X$.
Then $\pi $, regarded as a map from $X\rightarrow \pi (X)$, admits a
completely isometric inverse, which extends to a complete contraction $\psi
:B(H,K)/\mathcal{K}(H,K)\rightarrow B(H,K)$. Clearly, $\psi \circ \pi $ is a
completely contractive map which extends the inclusion of $X$ into $B(H,K)$,
but it does not extend the inclusion of $T$ into $B(H,K$), since it
annihilates the compact operators.

Conversely, suppose that the quotient map is not completely isometric on $X$%
. Let $\mathcal{S}(X) \subset B(K \oplus H)$ denote the Paulsen system
associated with $X$. We will verify that $\mathcal{S}(X)$ satisfies the
assumptions of Arveson's boundary theorem.

To see that $\mathcal{S}(X)$ acts irreducibly, suppose that $p$ is an
orthogonal projection on $K\oplus H$ which commutes with $\mathcal{S}(X)$.
In particular, $p$ commutes with $I_{K}\oplus 0$ and $0\oplus I_{H}$, from
which we deduce that $p=p_{1}\oplus p_{2}$, where $p_{1}\in B(K)$ and $%
p_{2}\in B(H)$ are orthogonal projections. For $x\in X$, we therefore have 
\begin{equation*}
\begin{bmatrix}
0 & p_{1}x \\ 
0 & 0%
\end{bmatrix}%
=%
\begin{bmatrix}
p_{1} & 0 \\ 
0 & p_{2}%
\end{bmatrix}%
\begin{bmatrix}
0 & x \\ 
0 & 0%
\end{bmatrix}%
=%
\begin{bmatrix}
0 & x \\ 
0 & 0%
\end{bmatrix}%
\begin{bmatrix}
p_{1} & 0 \\ 
0 & p_{2}%
\end{bmatrix}%
=%
\begin{bmatrix}
0 & xp_{2} \\ 
0 & 0%
\end{bmatrix}%
,
\end{equation*}%
hence $p_{1}x=xp_{2}$ for all $x\in X$. Since $X$ acts irreducibly, it
follows that either $p_{1}\oplus p_{2}=0$ or $p_{1}\oplus p_{2}=I_{K\oplus
H} $, so that $\mathcal{S}(X)$ acts irreducibly. The assumption that $T$
contains a non-zero compact operator implies that $\mathcal{S}(X)$ contains
a non-zero compact operator, hence by irreducibility of $\mathcal{S}(X)$, we
see that $\mathcal{K}(K\oplus H)\subset C^{\ast }(\mathcal{S}(X))$.

Since the quotient map $B(H,K)\rightarrow B(H,K)/\mathcal{K}(H,K)$ is not
completely isometric on $X$, there exists $x\in M_{n}(X)$ and $k\in M_{n}(%
\mathcal{K}(H,K))$ such that $||x-k||<||x||$. Regarding $x$ as an element of 
$M_{n}(\mathcal{S}(X))$ in the canonical way and correspondingly $k$ as an
element of $M_{n}(\mathcal{K}(K\oplus H))$, we see that the quotient map $%
B(K\oplus H)\rightarrow B(K\oplus H)/\mathcal{K}(K\oplus H)$ is not
completely isometric on $\mathcal S(X)$. Thus, Arveson's boundary theorem implies that the
identity representation is a boundary representation of $\mathcal{S}(X)$.
According to Proposition \ref{Proposition:Paulsen-boundary}, there exists a
boundary representation $\psi :X\rightarrow B(L_{1},L_{2})$ of $X$ such that 
$S(\psi )$ is the inclusion of $\mathcal{S}(X)$ into $B(K\oplus H)$. It
easily follows now that $L_{1}=H,L_{2}=K$ and that $\psi $ is the inclusion
of $X$ into $B(H,K)$, which finishes the proof.
\end{proof}

\subsection{Rectangular multipliers}

A \emph{reproducing kernel Hilbert space} $H$ on a set $X$ is a Hilbert
space of functions on $X$ such that for every $x\in X$, the functional 
\begin{equation*}
H\rightarrow \mathbb{C},\quad f\mapsto f(x),
\end{equation*}%
is bounded. The unique function $k:X\times X\rightarrow \mathbb{C}$ which
satisfies $k(\cdot ,x)\in H$ for all $x\in X$ and 
\begin{equation*}
\langle f,k(\cdot ,x)\rangle =f(x)
\end{equation*}
for all $x\in X$ and $f\in H$ is called the \emph{reproducing kernel} of $H$%
. We will always assume that $H$ has no common zeros, meaning that there
does not exist $x\in X$ such that $f(x)=0$ for all $x\in X$. Equivalently, $%
k(x,x)\neq 0$ for all $x\in X$. We refer the reader to the books \cite{PR16}
and \cite{AM02} for background material on reproducing kernel Hilbert spaces.

If $H$ and $K$ are reproducing kernel Hilbert spaces on the same set $X$, we
define the multiplier space 
\begin{equation*}
\Mult(H,K)=\{\varphi :X\rightarrow \mathbb{C}:\varphi \cdot f\in K\text{ for
all }f\in H\},
\end{equation*}%
where $\left( \varphi \cdot f\right) \left( x\right) =\varphi \left(
x\right) f\left( x\right) $ for $x\in X$; see \cite[Section 5.7]{PR16}. By 
\cite[Theorem 5.21]{PR16}, every $\varphi \in \Mult(H,K)$ induces a bounded
multiplication operator $M_{\varphi }:H\rightarrow K$. Moreover, since $K$
has no common zeros, every multiplier $\varphi $ is uniquely determined by
its associated multiplication operator $M_{\varphi }$. We may thus regard $%
\Mult(H,K)$ as a subspace of $B(H,K)$.

The best studied case occurs when $H=K$, in which case $\Mult(H)=\Mult(H,H)$
is an algebra, called the multiplier algebra of $H$; \cite[Section 2.3]{AM02}%
. Nevertheless, the rectangular case of two different reproducing kernel
Hilbert spaces has been studied as well, see for example \cite{Taylor66} and 
\cite{Stegenga80}, where multipliers between weighted Dirichlet spaces are
investigated.

We say that a reproducing kernel Hilbert space $K$ on $X$ with reproducing
kernel $k$ is \emph{irreducible} if $X$ cannot be partitioned into two
non-empty sets $X_{1}$ and $X_{2}$ such that $k(x,y)=0$ for all $x\in X_{1}$
and $y\in X_{2}$. This definition is more general than the definition of
irreducibility in \cite[Definition 7.1]{AM02}, but it suffices for our
purposes.

For the next Lemma, we observe that if $H$ contains the constant function $1$%
, then $\Mult(H,K)$ is contained in $K$.

\begin{lemma}
\label{L:mult_irred} Let $H$ and $K$ be reproducing kernel Hilbert spaces on
the same set. Suppose that

\begin{itemize}
\item $H$ contains the constant function $1$,

\item $\Mult(H,K)$ is dense in $K$, and

\item $K$ is irreducible.
\end{itemize}

Then $\Mult(H,K) \subset B(H,K)$ acts irreducibly.
\end{lemma}

\begin{proof}
Suppose that $p\in B(H)$ and $q\in B(K)$ are orthogonal projections which
satisfy $qM_{\varphi }=M_{\varphi }p$ for all $\varphi \in \Mult(H,K)$.
Define $\psi =p1\in H$. Then for all $\varphi \in \Mult(H,K)$, the identity 
\begin{equation*}
q\varphi =qM_{\varphi }1=M_{\varphi }p1=\psi \varphi
\end{equation*}%
holds. Since $\Mult(H,K)$ is dense in $K$, we deduce that $\psi \in \Mult(K)$
and that $q=M_{\psi }$. We claim that $\psi $ is necessarily constant. To
this end, let $k$ denote the reproducing kernel of $K$. Note that $M_{\psi }$
is in particular selfadjoint, so that 
\begin{equation*}
\psi (x)k(x,y)=\langle M_{\psi }k(\cdot ,y),k(\cdot ,x)\rangle =\langle
k(\cdot ,y),M_{\psi }k(\cdot ,x)\rangle =\overline{\psi (y)}k(x,y)
\end{equation*}%
for all $x,y\in X$. Hence $\psi (x)$ is real for all $x\in X$ and $\psi
(x)=\psi (y)$ if $k(x,y)\neq 0$. Fix $x_{0}\in X$, and suppose for a
contradiction that 
\begin{equation*}
X_{1}=\{x\in X:\psi (x)=\psi (x_{0})\}
\end{equation*}%
is a proper subset of $X$ and let $X_{2}=X\setminus X_{1}$. If $x\in X_{1}$
and $y\in X_{2}$, then $\psi (y)\neq \psi (x_{0})=\psi (x)$, hence $k(x,y)=0$%
. This contradicts irreducibility of $K$, so that $\psi $ is constant.
Moreover, since $M_{\psi }$ is a projection, we necessarily have $\psi =1$
or $\psi =0$. If $\psi =1$, then $q=M_{\psi }=I_{K}$ and $M_{\varphi
}(I_{H}-p)=0$ for all $\varphi \in \Mult(H,K)$. Similarly, if $\psi =0$,
then $q=0$ and $M_{\varphi }p=0$ for all $\varphi \in \Mult(H,K)$. We may
thus finish the proof by showing that 
\begin{equation*}
\bigcap_{\varphi \in \Mult(H,K)}\ker (M_{\varphi })=\{0\}.
\end{equation*}%
To this end, note that if $x\in X$, then $\{f\in K:f(x)=0\}$ is a proper
closed subspace of $K$, as $K$ has no common zeros. Since $\Mult(H,K)$ is
dense in $K$, it cannot be contained in such a subspace, thus for every $%
x\in X$, there exists $\varphi \in \Mult(H,K)$ such that $\varphi (x)\neq 0$%
. If $f\in H$ satisfies $M_{\varphi }f=0$ for all $\varphi \in \Mult(H,K)$,
it therefore follows that $f(x)=0$ for all $x\in X$, that is, $f=0$, as
desired.
\end{proof}

We can now use the rectangular boundary theorem to show that for many
multiplier spaces, the identity representation is always a boundary
representation.

\begin{proposition}
\label{prop:rectangular_mult_boundary} Let $H$ and $K$ be reproducing kernel
Hilbert spaces on the same set and let $M = \Mult(H,K)$. Suppose that

\begin{itemize}
\item $H$ contains the constant function $1$,

\item $M$ is dense in $K$,

\item $K$ is irreducible, and

\item $M$ contains a non-zero compact operator.
\end{itemize}

Then the identity representation is a boundary representation of $M $. In
particular, the triple envelope of $M$ is the TRO generated by $M$.
\end{proposition}

\begin{proof}
Lemma \ref{L:mult_irred} shows that $M$ acts irreducibly. Moreover, the
quotient map by the compacts is not isometric on $M$ since $M$ contains a
non-zero compact operator. An application of the rectangular boundary
theorem (Theorem \ref{T:rectangular_boundary_thm}) now finishes the proof.
\end{proof}

For $s\in \mathbb{R}$, let 
\begin{equation*}
H_{s}=\Big\{f(z)=\sum_{n=0}^{\infty
}a_{n}z^{n}:||f||_{H_s^2}=\sum_{n=0}^{\infty }|a_{n}|^{2}(n+1)^{-s}<\infty %
\Big\}.
\end{equation*}%
This is a reproducing kernel Hilbert space on the open unit disc $\mathbb{D}$
with reproducing kernel 
\begin{equation*}
k_{s}(z,w)=\sum_{n=0}^{\infty }(n+1)^{s}(z\overline{w})^{n}.
\end{equation*}%
This scale of spaces is a frequent object of study in the theory of
reproducing kernel Hilbert spaces. The space $H_{0}$ is the classical Hardy
space $H^{2}$, the space $H_{-1}$ is the Dirichlet space, and the space $%
H_{1}$ is the Bergman space.

The elements of $\Mult(H_{s},H_{t})$ were characterized in \cite{Taylor66}
and \cite{Stegenga80}. We remark that the spaces $\mathcal{D}_{\alpha }$ of 
\cite{Taylor66} are related to the spaces above via the formula $\mathcal{D}%
_{\alpha }=H_{-\alpha }$. In \cite{Stegenga80}, a slightly different
convention is used. There, $\mathcal{D}_{\alpha }=H_{-2\alpha }$, at least
with equivalent norms. Theorem 4 of \cite{Taylor66} shows that $\Mult%
(H_{s},H_{t})=\{0\} $ if $s>t$. On the other hand, if $s\leq t$, then $%
H_{s}\subset H_{t}$, hence $\Mult(H_{s})\subset \Mult(H_{s},H_{t})$. Since $%
\Mult(H_{s})$ at least contains the polynomials, the same is true for $\Mult%
(H_{s},H_{t})$.

In the square case $s=t$, boundary representations of operator spaces
related to the algebras $\Mult(H_{s})$, and their analogs on higher
dimensional domains, were studied in \cite{GHX04,KS13,CH}; see in particular 
\cite[Section 2]{GHX04} and \cite[Section 5.2]{KS13}. It is well known that
if $s\geq 0$, then $\Mult(H_{s})=H^{\infty }$, the algebra of all bounded
analytic functions on the unit disc, endowed with the supremum norm. This
can be deduced, for example, from \cite[Proposition 26 (ii)]{Shields74}. In
particular, the C*-envelope of $\Mult(H_{s})$ is commutative, so that the
identity representation of $\Mult(H_{s})$ on $H_{s}$ is not a boundary
representation. On the other hand, if $s<0$, then the identity
representation of $\Mult(H_{s})$ on $H_{s}$ is a boundary representation.
This follows, for instance, from Corollary 2 in Section 2 of \cite{Arveson72}
and its proof.

We use the results above to prove that, in the rectangular case, the
identity representation is always a boundary representation.

\begin{corollary}
If $s < t$, then the identity representation is a boundary representation of 
$\Mult(H_s,H_t)$.
\end{corollary}

\begin{proof}
We verify that the pair $(H_{s},H_{t})$ satisfies the assumptions of
Proposition \ref{prop:rectangular_mult_boundary}. It is clear that $H_{s}$
contains the constant function $1$. By the remark above, $\Mult(H_{s},H_{t})$
contains the polynomials, and is therefore dense in $K$. Moreover, since $%
k_{t}(0,w)=1$ for all $w\in \mathbb{D}$, the space $H_{t}$ is irreducible.
Finally, since 
\begin{equation*}
\frac{||z^{n}||_{H_t}^{2}}{||z^{n}||_{H_{s}}^{2}}=(n+1)^{s-t},
\end{equation*}%
which tends to zero as $n\rightarrow \infty $, the inclusion $H_{s}\subset
H_{t}$ is compact, so that $M_{1}\in \Mult(H_{s},H_{t})$ is a compact
operator. Therefore, the result follows from Proposition \ref%
{prop:rectangular_mult_boundary}.
\end{proof}

\section{Operator spaces and rectangular matrix convex sets\label%
{Sec:rectangular}}

In the following we will use notation from \cite{effros_matrix_1997} and 
\cite{webster_krein-milman_1999}. In particular, if $V$ and $V^{\prime }$
are vector spaces in duality via a bilinear map $\left\langle \cdot ,\cdot
\right\rangle $, $x=\left[ x_{ij}\right] \in M_{n,m}\left( V\right) $, and $%
\psi =\left[ \psi _{\alpha \beta }\right] \in M_{r,s}\left( V^{\prime
}\right) $, then we let $\left\langle \left\langle x,\psi \right\rangle
\right\rangle $ be the element $\left[ \left\langle x_{ij},\psi _{\alpha
\beta }\right\rangle \right] $ of $M_{nr,ms}\left( \mathbb{C}\right) $,
where the rows of $\left\langle \left\langle x,\psi \right\rangle
\right\rangle $ are indexed by $\left( i,\alpha \right) $ and the columns of 
$\left\langle \left\langle x,\psi \right\rangle \right\rangle $ are indexed
by $\left( j,\beta \right) $. We also let $\psi ^{\left( n,m\right) }$ be
the map $\left\langle \left\langle \cdot ,\psi \right\rangle \right\rangle
:M_{n,m}\left( V\right) \rightarrow M_{nr,ms}\left( \mathbb{C}\right) $.

\subsection{Rectangular matrix convex sets\label{Subsection:convex}}

\begin{definition}
A \emph{rectangular matrix convex set }in a vector space $V$ is a collection 
$\boldsymbol{K}=\left( K_{n,m}\right) $ of subsets of $M_{n,m}\left(
V\right) $ with the property that for any $\alpha _{i}\in M_{n_{i},n}\left( 
\mathbb{C}\right) $ and $\beta _{i}\in M_{m_{i},m}\left( \mathbb{C}\right) $
and $v_{i}\in K_{n_{i},m_{i}}$ for $1\leq i\leq \ell $ such that 
\begin{equation*}
\left\Vert \alpha _{1}^{\ast }\alpha _{1}+\cdots +\alpha _{\ell }^{\ast
}\alpha _{\ell }\right\Vert \left\Vert \beta _{1}^{\ast }\beta _{1}+\cdots
+\beta _{\ell }^{\ast }\beta _{\ell }\right\Vert \leq 1
\end{equation*}%
one has that $\alpha _{1}^{\ast }v_{1}\beta _{1}+\cdots +\alpha _{\ell
}^{\ast }v_{\ell }\beta _{\ell }\in K_{n,m}$.
\end{definition}

When $V$ is a topological vector space, we say that $\boldsymbol{K}$ is
compact if $K_{n,m}$ is compact for every $n,m$. The following
characterization of rectangular matrix convex sets can be easily verified
using Proposition \ref{Proposition:dilation-TRO} and the fact that any
finite-dimensional representation of $M_{n,m}\left( \mathbb{C}\right) $ as a
TRO is unitarily conjugate to a finite direct sum of copies of the identity
representation \cite[Lemma 3.2.3]{bohle_k-theory_2011}.

\begin{lemma}
Suppose that $\boldsymbol{K}=\left( K_{n,m}\right) $ where $K_{n,m}\subset
M_{n,m}\left( V\right) $. The following assertions are equivalent:

\begin{enumerate}
\item $\boldsymbol{K}$ is a rectangular convex set;

\item $x\oplus y\in K_{n+m,r+s}$ for any $x\in K_{n,r}$ and $y\in K_{m,s}$,
and $\alpha ^{\ast }x\beta \in K_{r,s}$ for any $x\in K_{n,m}$, $\alpha \in
M_{n,r}\left( \mathbb{C}\right) $ and $\beta \in M_{m,s}\left( \mathbb{C}%
\right) $ with $\left\Vert \alpha ^{\ast }\alpha \right\Vert \left\Vert
\beta ^{\ast }\beta \right\Vert \leq 1$;

\item $x\oplus y\in K_{n+m,r+s}$ for any $x\in K_{n,r}$ and $y\in K_{m,s}$,
and $\left( \sigma \otimes id_{V}\right) \left[ K_{n,m}\right] \subset
K_{r,s}$ for any completely contractive map $\sigma :M_{n,m}\left( \mathbb{C}%
\right) \rightarrow M_{r,s}\left( \mathbb{C}\right) $.
\end{enumerate}
\end{lemma}

It is clear that, if $\boldsymbol{K}$ is a rectangular matrix convex set,
then $\left( K_{n,n}\right) $ is a matrix convex set in the sense of \cite%
{wittstock_matrix_1984}. Furthermore if $\boldsymbol{K}$ and\textbf{\ }$%
\boldsymbol{T}$ are rectangular matrix convex sets such that $T_{n}=K_{n}$
for every $n\in \mathbb{N}$ then $T_{n,m}=K_{n,m}$ for every $n,m\in \mathbb{%
N}$. If $\boldsymbol{S}=\left( S_{n,m}\right) $ is a collection of subsets
of a (topological) vector space $V$, the (closed) rectangular matrix convex
hull of $\boldsymbol{S}$ is the smallest (closed) rectangular matrix convex
set containing $\boldsymbol{S}$.

\begin{example}
\label{Example:complete-ball}Suppose that $X$ is an operator space. Set $%
K_{n,m}$ to be space of completely contractive maps from $X$ to $%
M_{n,m}\left( \mathbb{C}\right) $. Then $\mathrm{C\mathrm{Ball}}(X)=\left(
K_{n,m}\right) $ is a rectangular matrix convex set.
\end{example}

\subsection{The rectangular polar theorem\label{Subsection:polar}}

Suppose that $V$ and $V^{\prime }$ are vector spaces in duality. We endow
both $V$ and $V^{\prime }$ with the weak topology induced from such a
duality. Let $\boldsymbol{S}=\left( S_{n,m}\right) $ be a collection of
subsets $S_{n,m}\subset M_{n,m}\left( V\right) $. We define the \emph{%
rectangular matrix polar} $\boldsymbol{S}^{\rho }$ to be the closed
rectangular matrix convex subset of $V^{\prime }$ such that $f\in
S_{n,m}^{\rho }$ if and only if $\left\Vert \left\langle \left\langle
v,f\right\rangle \right\rangle \right\Vert \leq 1$ for every $r,s\in \mathbb{%
N}$ and every $v\in S_{r,s}$. The same proof as \cite[Lemma 5.1]%
{effros_matrix_1997} shows that $f\in S_{n,m}^{\rho }$ if and only if $%
\left\Vert \left\langle \left\langle v,f\right\rangle \right\rangle
\right\Vert \leq 1$ for every $v\in S_{n,m}$.

If $A\subset V$, then its \emph{absolute polar} $A^{\circ }$ is the set of $%
f\in V^{\prime }$ such that $\left\vert \left\langle v,f\right\rangle
\right\vert \leq 1$ for every $v\in A$. The classical bipolar theorem
asserts that the absolute bipolar $A^{\circ \circ }$ is the closed
absolutely convex hull of $A$ \cite[Theorem 8.1.12]{conway_course_1990}. We
will prove below the rectangular analog of this fact. The proof is analogous
to the one of \cite[Theorem 5.4]{effros_matrix_1997}.

\begin{theorem}
If $\boldsymbol{S}=\left( S_{n,m}\right) $ is a collection of subsets $%
S_{n,m}\subset M_{n,m}\left( V\right) $, then the rectangular matrix bipolar 
$\boldsymbol{S}^{\rho \rho }$ is the closed rectangular matrix convex hull
of $\boldsymbol{S}$.
\end{theorem}

The proof of \cite[Theorem C]{effros_abstract_1993} shows that if $%
\boldsymbol{K}$ is a rectangular matrix convex set in a vector space $V$,
and $F$ is a linear functional on $M_{n,m}\left( V\right) $ satisfying $%
\left\Vert F|_{K_{n,m}}\right\Vert \leq 1$, then

\begin{enumerate}
\item there exist states $p$ on $M_{n}\left( \mathbb{C}\right) $ and $q$ on $%
M_{m}\left( \mathbb{C}\right) $ such that $\left\vert F\left( \alpha ^{\ast
}v\beta \right) \right\vert ^{2}\leq p\left( \alpha ^{\ast }\alpha \right)
q\left( \beta ^{\ast }\beta \right) $ for every $r,s\in \mathbb{N}$, $\alpha
\in M_{n,r}\left( \mathbb{C}\right) $, $\beta \in M_{m,s}\left( \mathbb{C}%
\right) $, and $v\in M_{r,s}\left( V\right) $, and

\item there exist matrices $\gamma \in M_{n^{2},1}\left( \mathbb{C}\right) $%
, $\delta \in M_{m^{2},1}\left( \mathbb{C}\right) $, and a map $\varphi
:V\rightarrow M_{n,m}\left( \mathbb{C}\right) $, such that $F\left( w\right)
=\gamma ^{\ast }\left\langle \left\langle w,\varphi \right\rangle
\right\rangle \delta $ for every $w\in M_{n,m}\left( W\right) $ and $%
\left\Vert \left\langle \left\langle w,\varphi \right\rangle \right\rangle
\right\Vert \leq 1$ for every $r,s\in \mathbb{N}$ and $w\in K_{r,s}$.
\end{enumerate}

From this one can easily deduce the following proposition, which gives the
rectangular matrix bipolar theorem as an easy consequence.

\begin{proposition}
Suppose that $V$ and $V^{\prime }$ are vector spaces in duality, and $%
\boldsymbol{K}$ is a compact rectangular convex space in $V$. If $v_{0}\in
M_{n,m}\left( V\right) \backslash K_{n,m}$, then there exists $\varphi \in
K_{n,m}^{\rho }$ such that $\left\Vert \left\langle \left\langle
v_{0},\varphi \right\rangle \right\rangle \right\Vert >1$.
\end{proposition}

\begin{proof}
By the classical bipolar theorem there exists a continuous linear functional 
$F$ on $M_{n,m}\left( V\right) $ such that $\left\Vert
F|_{K_{n,m}}\right\Vert \leq 1$ and $\left\vert F\left( v_{0}\right)
\right\vert >1$. By the remarks above there exists $\varphi \in
K_{n,m}^{\rho }$ and contractive $\gamma \in M_{n^{2}\times 1}\left( \mathbb{%
C}\right) $ and $\delta \in M_{m^{2}\times 1}\left( \mathbb{C}\right) $ such
that $F(v)=\gamma ^{\ast }\left\langle \left\langle v,\varphi \right\rangle
\right\rangle \delta $. Thus we have $\left\Vert \left\langle \left\langle
v_{0},\varphi \right\rangle \right\rangle \right\Vert \geq \left\Vert \gamma
^{\ast }\left\langle \left\langle v_{0},\varphi \right\rangle \right\rangle
\delta \right\Vert =\left\Vert F\left( v_{0}\right) \right\Vert >1$.
\end{proof}

\subsection{Representation of rectangular convex sets\label%
{Subsection:representation}}

Suppose that $\boldsymbol{K}$ is a rectangular matrix convex set in a vector
space $V$. A \emph{rectangular matrix convex combination} in a rectangular
convex set $\boldsymbol{K}$ is an expression of the form $\alpha _{1}^{\ast
}v_{1}\beta _{1}+\cdots +\alpha _{\ell }^{\ast }v_{\ell }\beta _{\ell }$ for 
$v_{i}\in K_{n_{i},m_{i}}$, $\alpha _{i}\in M_{n_{i},n}\left( \mathbb{C}%
\right) $, and $\beta _{i}\in M_{m_{i},m}\left( \mathbb{C}\right) $ such
that $\alpha _{1}^{\ast }\alpha _{1}+\cdots +\alpha _{\ell }^{\ast }\alpha
_{\ell }=1$, and $\beta _{1}^{\ast }\beta _{1}+\cdots +\beta _{\ell }^{\ast
}\beta _{\ell }=1$. A\emph{\ proper rectangular matrix convex combination}
is a rectangular convex combination $\alpha _{1}^{\ast }v_{1}\beta
_{1}+\cdots +\alpha _{\ell }^{\ast }v_{\ell }\beta _{\ell }$ where
furthermore $\alpha _{1},\ldots ,\alpha _{\ell }$ and $\beta _{1},\ldots
,\beta _{\ell }$ are right invertible. Observe that these notions are a
particular instance of the notions of (proper) rectangular operator convex
combination introduced in Subsection \ref{Subs:rectangular}.

\begin{definition}
A \emph{rectangular matrix affine mapping }from a rectangular convex set $%
\boldsymbol{K}$ to a rectangular convex set $\mathbf{T}$ is a sequence $%
\boldsymbol{\theta }$ of maps $\theta _{n,m}:K_{n,m}\rightarrow T_{n,m}$
that preserves rectangular matrix convex combinations.
\end{definition}

When $\boldsymbol{K}$ and $\boldsymbol{T}$ are compact rectangular convex
sets, we say that $\boldsymbol{\theta }$ is continuous (respectively, a
homeomorphism) when $\theta _{n,m}$ is continuous (respectively, a
homeomorphism) for every $n,m\in \mathbb{N}$.

Given a compact rectangular matrix convex set\emph{\ }$\boldsymbol{K}$ we
let $A_{\rho }(\boldsymbol{K})$ be the complex vector space of continuous
rectangular matrix affine mappings from $\boldsymbol{K}$ to $\mathrm{C%
\mathrm{Ball}}\left( \mathbb{C}\right) $. Here $\mathrm{C\mathrm{Ball}}%
\left( \mathbb{C}\right) $ is the compact rectangular matrix convex set
defined as in Example \ref{Example:complete-ball}, where $\mathbb{C}$ is
endowed with its canonical operator space structure. The space $A_{\rho }(%
\boldsymbol{K})$ has a natural operator space structure where $M_{n,m}\left(
A_{\rho }(\boldsymbol{K})\right) $ is identified isometrically with a
subspace of $C\left( K_{n,m},M_{n,m}\left( \mathbb{C}\right) \right) $
endowed with the supremum norm. More generally if $Y$ is any operator space,
then we define $A_{\rho }(\boldsymbol{K},Y)$ to be the operator space of
continuous rectangular affine mappings from $\boldsymbol{K}$ to $\mathrm{C%
\mathrm{Ball}}(Y)$. Observe that $M_{n,m}\left( A_{\rho }(\boldsymbol{K}%
)\right) $ is completely isometric to $A_{\rho }(\boldsymbol{K})$.

Starting from the operator space $A_{\rho }(\boldsymbol{K})$ one can
consider the compact rectangular matrix convex set $\mathrm{C\mathrm{Ball}}%
(A_{\rho }(\boldsymbol{K})^{\prime })$ as in Example \ref%
{Example:complete-ball}. Here $A_{\rho }(\boldsymbol{K})^{\prime }$ denotes
the dual space of the operator space $A_{\rho }(\boldsymbol{K})$, endowed
with its canonical operator space structure. There is a canonical
rectangular matrix affine mapping $\boldsymbol{\theta }$ from $\boldsymbol{K}
$ to $\mathrm{C\mathrm{Ball}}(A_{\rho }(\boldsymbol{K})^{\prime })$ given by
point evaluations. It is clear that this map is injective. It is furthermore
surjective in view of the rectangular bipolar theorem. The argument is
similar to the one of the proof of \cite[Proposition 3.5]%
{webster_krein-milman_1999}. This shows that the map $\boldsymbol{\theta }$
is indeed a rectangular matrix affine homeomorphism from $\boldsymbol{K}$
onto $\mathrm{C\mathrm{Ball}}(A_{\rho }(\boldsymbol{K})^{\prime })$. This
implies that the assignment $X\mapsto \mathrm{C\mathrm{Ball}}(X^{\prime })$
is a 1:1 correspondence between operator spaces and rectangular convex sets.
It is also not difficult to verify that this correspondence is in fact an
equivalence of categories, where morphisms between operator spaces are
completely contractive linear maps, and morphisms between rectangular convex
sets are continuous rectangular matrix affine mappings.

\subsection{The rectangular Krein-Milman theorem\label{Subsection:Krein}}

The notion of (proper) rectangular convex combination yields a natural
notion of extreme point in a rectangular convex set. An element $v$ of a
rectangular convex set $\boldsymbol{K}$ is a \emph{rectangular matrix extreme%
} \emph{point} if for any proper rectangular convex combination $\alpha
_{1}^{\ast }v_{1}\beta _{1}+\cdots +\alpha _{\ell }^{\ast }v_{\ell }\beta
_{\ell }=v$ for $v_{i}\in K_{n_{i},m_{i}}$ one has that, for every $1\leq
i\leq \ell $, $n_{i}=n$, $m_{i}=m$, and $v_{i}=u_{i}^{\ast }vw_{i}$ for some
unitaries $u_{i}\in M_{n}\left( \mathbb{C}\right) $ and $w_{i}\in K_{m}$. We
now observe that the notion of rectangular extreme point coincides with the
notion of rectangular operator extreme operator state from Definition \ref%
{Definition:rectangular-operator-extreme}. The argument is borrowed from the
proof of \cite[Theorem\ B]{farenick_extremal_2000}.

\begin{lemma}
\label{Lemma:rectangular-extreme}Suppose that $X$ is an operator space, $%
\boldsymbol{K}=\mathrm{C\mathrm{Ball}}(X)$, and $\phi \in K_{n,m}$. Then $%
\phi $ is a rectangular matrix extreme point of $\boldsymbol{K}$ if and only
if it is a rectangular operator extreme operator state of $X$.
\end{lemma}

\begin{proof}
It is clear that a rectangular operator extreme point is a rectangular
matrix extreme point. We prove the converse implication. Suppose that $\phi $
is a rectangular matrix extreme point. Let $\phi =\alpha _{1}^{\ast }\phi
_{1}\beta _{1}+\cdots +\alpha _{\ell }^{\ast }\phi _{\ell }\beta _{\ell }$
be a proper rectangular matrix convex combination, where $\phi _{i}\in
K_{n_{i},m_{i}}$ for $i=1,2,\ldots ,\ell $. By assumption, we have that $%
n_{i}=n$ and $m_{i}=m$ for $i=1,2,\ldots ,m$, and there exist unitaries $%
u_{i}\in M_{n}\left( \mathbb{C}\right) $ and $w_{i}\in M_{m}\left( \mathbb{C}%
\right) $ such that $\phi _{i}=u_{i}^{\ast }\phi w_{i}$ for $i=1,2,\ldots
,\ell $. Therefore we have that%
\begin{equation}
\phi =\left( u_{1}\alpha _{1}\right) ^{\ast }\phi \left( w_{1}\beta
_{1}\right) +\cdots +\left( u_{\ell }\alpha _{\ell }\right) ^{\ast }\phi
_{\ell }\left( w_{\ell }\beta _{\ell }\right) \text{.\label%
{Equation:matrix-extreme}}
\end{equation}%
Define $R\subset M_{m+m}\left( \mathbb{C}\right) $ to be the range of $%
\mathcal{S}(\phi )$. Observe that it follows from the fact that $\phi $ is a
rectangular extreme point that the commutant of $R$ is one-dimensional. Set%
\begin{equation*}
A_{i}=%
\begin{bmatrix}
u_{i}\alpha _{i} & 0 \\ 
0 & w_{i}\beta _{i}%
\end{bmatrix}%
\end{equation*}%
for $i=1,2,\ldots ,n$. Define the unital completely positive map $\Psi
:M_{n+m}\left( \mathbb{C}\right) \rightarrow M_{n+m}\left( \mathbb{C}\right) 
$, $z\mapsto A_{1}^{\ast }zA_{1}+\cdots +A_{\ell }^{\ast }zA_{\ell }$. By
Equation \eqref{Equation:matrix-extreme} we have that $\Psi (z)=z$ for every 
$z\in R$. It follows from this and \cite[Theorem 2.11]%
{arveson_subalgebras_1972} that $\Psi (z)=z$ for every $z\in M_{n+m}\left( 
\mathbb{C}\right) $. By the uniqueness statement in the Choi's
representation of a unital completely positive map \cite%
{choi_completely_1975}, we deduce that there exist $\lambda _{i}\in \mathbb{C%
}$ such that $A_{i}=\lambda _{i}1$ for $i=1,2,\ldots ,\ell $. Therefore $%
\alpha _{i}^{\ast }\alpha _{i}=\left( u_{i}\alpha _{i}\right) ^{\ast }\left(
u_{i}\alpha _{i}\right) =\left\vert \lambda _{i}\right\vert ^{2}1$, $\beta
_{i}^{\ast }\beta _{i}=\left( w_{i}\beta _{i}\right) ^{\ast }\left(
w_{i}\beta _{i}\right) =\left\vert \lambda _{i}\right\vert ^{2}1$, and $%
\alpha _{i}^{\ast }\phi _{i}\beta _{i}=\left( u_{i}\alpha _{i}\right) ^{\ast
}\phi \left( w_{i}\beta _{i}\right) =\left\vert \lambda _{i}\right\vert
^{2}\phi $. This concludes the proof that $\phi $ is a rectangular operator
extreme point.
\end{proof}

We denote by $\partial _{\rho }\boldsymbol{K}=\left( \partial _{\rho
}K_{n,m}\right) $ set of rectangular matrix extreme points of $\boldsymbol{K}
$. Recall that the Krein-Milman theorem asserts that, if $K\subset V$ is a
compact convex subset of a topological vector space $V$, then $K$ is the
closed convex hull of the set of its extreme points. The following is the
natural analog of the Krein-Milman theorem for compact rectangular matrix
convex sets. The proof is analogous to the proof of the Krein-Milan theorem
for compact matrix convex sets \cite[Theorem 4.3]{webster_krein-milman_1999}.

\begin{theorem}
\label{Theorem:Krein-Milman}Suppose that $\boldsymbol{K}$ is a compact
rectangular convex set. Then $\boldsymbol{K}$ is the closed rectangular
matrix convex hull of $\partial _{\rho }\boldsymbol{K}$.
\end{theorem}

\begin{proof}
Suppose that $\boldsymbol{K}$ is a compact rectangular convex set. In view
of the representation theorem from Subsection \ref{Subsection:representation}%
, we can assume without loss of generality that $\boldsymbol{K}=\mathrm{C%
\mathrm{Ball}}(X^{\prime })$ for some operator space $X$. We will assume
that $X$ is concretely represented as a subspace of $B(H)$ for some Hilbert
space $H$. We will also canonically identify $M_{n,r}(X^{\prime })$ with the
space of bounded linear functionals on $M_{n,r}(X)$.

Fix $n,m\in \mathbb{N}$. Let $\tilde{X}$ be the space of operators of the
form 
\begin{equation*}
\begin{bmatrix}
\lambda I^{\oplus n} & x \\ 
y^{\ast } & \mu I^{\oplus m}%
\end{bmatrix}%
\end{equation*}%
for $\lambda ,\mu \in \mathbb{C}$ and $x,y\in M_{n,m}(X)$, where $I^{\oplus
n}$ and $I^{\oplus m}$ are the identity operator on, respectively, the $n$%
-fold and $m$-fold Hilbertian sum of $H$ by itself. If $\varphi \in
M_{r,s}(X^{\prime })$, then we denote by $\tilde{\varphi}\in M_{nr+ms}(%
\tilde{X}^{\prime })$ the element defined by%
\begin{equation*}
\begin{bmatrix}
\lambda I^{\oplus n} & x \\ 
y^{\ast } & \mu I^{\oplus n}%
\end{bmatrix}%
\mapsto 
\begin{bmatrix}
\lambda I_{rn} & \varphi ^{(n,m)}(x) \\ 
\varphi ^{(n,m)}(y)^{\ast } & \mu I_{ms}%
\end{bmatrix}%
\end{equation*}%
where $I_{rn}$ and $I_{ms}$ denote the identity $rn\times rn$ and $ms\times
ms$ matrices. If $\xi \in M_{r,n}\left( \mathbb{C}\right) $ and $\eta \in
M_{s,m}\left( \mathbb{C}\right) $ we also set%
\begin{equation*}
\xi \odot \eta :=%
\begin{bmatrix}
I_{n}\otimes \xi & 0 \\ 
0 & I_{m}\otimes \eta%
\end{bmatrix}%
\text{.}
\end{equation*}%
We let $\Delta _{n,m}$ be the set of elements of $M_{n^{2}+m^{2}}(\tilde{X}%
^{\prime })$ of the form $\left( \xi \odot \eta \right) ^{\ast }\tilde{%
\varphi}\left( \xi \odot \eta \right) $ for $r,s\in \mathbb{N}$, $\varphi
\in K_{r,s}$, $\xi \in M_{r,n}\left( \mathbb{C}\right) $ and $\eta \in
M_{s,m}\left( \mathbb{C}\right) $ such that $\left\Vert \xi \right\Vert
_{2}=\left\Vert \eta \right\Vert _{2}=1$. It is not difficult to verify as
in \cite[\S 4]{webster_krein-milman_1999} that one can assume without loss
of generality that $r\leq n$, $s\leq m$, and $\xi ,\eta $ are right
invertible. The computation below shows that $\Delta _{n,m}$ is convex. If $%
t_{1},t_{2}\in \left[ 0,1\right] $ are such that $t_{1}+t_{2}=1$ then%
\begin{equation*}
t_{1}\left( \xi _{1}\odot \eta _{1}\right) ^{\ast }\tilde{\varphi}_{1}\left(
\xi _{1}\odot \eta _{1}\right) +t_{2}\left( \xi _{1}\odot \eta _{1}\right)
^{\ast }\tilde{\varphi}_{2}\left( \xi _{2}\odot \eta _{2}\right) =\left( \xi
\odot \eta \right) ^{\ast }\varphi \left( \xi \odot \eta \right)
\end{equation*}%
where%
\begin{equation*}
\xi =%
\begin{bmatrix}
t_{1}\xi _{1} \\ 
t_{2}\xi _{2}%
\end{bmatrix}%
\text{, }\eta =%
\begin{bmatrix}
t_{1}\eta _{1} \\ 
t_{2}\eta _{2}%
\end{bmatrix}%
\text{, and }\varphi =%
\begin{bmatrix}
\varphi _{1} & 0 \\ 
0 & \varphi _{2}%
\end{bmatrix}%
\text{.}
\end{equation*}%
Thus $\Delta _{n,m}$ is a compact convex subset of the space of unital
completely positive maps from $\tilde{X}$ to $M_{n^{2}+m^{2}}\left( \mathbb{C%
}\right) $. Consider now an element $\left( \xi \odot \eta \right) ^{\ast }%
\tilde{\varphi}\left( \xi \odot \eta \right) $ of $\Delta _{n,m}$, where $%
\xi \in M_{r,n}\left( \mathbb{C}\right) $ and $\eta \in M_{s,m}\left( 
\mathbb{C}\right) $ are right invertible and $\varphi \in K_{n,m}$. Assume
that $\left( \xi \odot \eta \right) ^{\ast }\tilde{\varphi}\left( \xi \odot
\eta \right) $ is an extreme point of $\Delta _{n,m}$. We claim that this
implies that $\varphi $ is a rectangular extreme point of $\boldsymbol{K}$.
Indeed suppose that, for some $s_{k},r_{k}\in \mathbb{N}$, $\varphi _{k}\in
K_{r_{k},s_{k}}$, $\delta _{k}\in M_{s_{k},s}\left( \mathbb{C}\right) $, and 
$\gamma _{k}\in M_{r_{k},r}\left( \mathbb{C}\right) $, $\gamma _{1}^{\ast
}\varphi _{1}\delta _{1}+\cdots +\gamma _{\ell }^{\ast }\varphi _{\ell
}\delta _{\ell }$ is a proper rectangular convex combination in $\boldsymbol{%
K}$ that equals $\varphi$. Then we have that%
\begin{equation*}
\left( \xi \odot \eta \right) ^{\ast }\tilde{\varphi}\left( \xi \odot \eta
\right) =\left( \gamma _{1}\xi \odot \delta _{1}\eta \right) ^{\ast }\tilde{%
\varphi}_{1}\left( \gamma _{1}\xi \odot \delta _{1}\eta \right) +\cdots
+\left( \gamma _{1}\xi \odot \delta _{1}\eta \right) ^{\ast }\tilde{\varphi}%
_{\ell }\left( \gamma _{1}\xi \odot \delta _{1}\eta \right) \text{.}
\end{equation*}%
Let 
\begin{equation*}
t_{k}=\left\Vert \left( \gamma _{k}\xi \odot \delta _{k}\eta \right)
\right\Vert _{2}
\end{equation*}%
for $k=1,2,\ldots ,\ell $. Observe that%
\begin{equation*}
\sum_{k=1}^{\ell }t_{k}^{2}=\frac{1}{r+s}\sum_{k=1}^{\ell }\left( r\mathrm{Tr%
}\left( \xi ^{\ast }\gamma _{k}^{\ast }\gamma _{k}\xi \right) +s\mathrm{Tr}%
\left( \eta ^{\ast }\delta _{k}^{\ast }\delta _{k}\eta \right) \right) =1%
\text{.}
\end{equation*}%
Therefore, if we set 
\begin{equation*}
\psi _{k}:=t_{k}^{-2}\left( \xi \gamma _{1}\odot \eta \delta _{1}\right)
^{\ast }\tilde{\varphi}_{k}\left( \xi \gamma _{1}\odot \eta \delta
_{1}\right)
\end{equation*}%
for $k=1,2,\ldots ,\ell $, we obtain elements $\psi _{1},\ldots ,\psi _{\ell
}$ of $\Delta _{n,m}$ such that $t_{1}^{2}\psi _{1}+\cdots +t_{\ell
}^{2}\psi _{\ell }=\left( \xi \odot \eta \right) ^{\ast }\tilde{\varphi}%
\left( \xi \odot \eta \right) $. Since by assumption $\left( \xi \odot \eta
\right) ^{\ast }\tilde{\varphi}\left( \xi \odot \eta \right) $ is an extreme
point of $\Delta _{n,m}$, we can conclude that $\psi _{k}=\left( \xi \odot
\eta \right) ^{\ast }\tilde{\varphi}\left( \xi \odot \eta \right) $ for $%
k=1,2,\ldots ,\ell $. The fact that $\xi $ and $\eta $ are right invertible
now easily implies that $r_{k}=r$, $s_{k}=s$, $\gamma _{i}^{\ast }\gamma
_{i}=t_{i}^{2}1$, $\delta _{i}^{\ast }\delta _{i}=t_{i}^{2}1$, and $\gamma
_{i}^{\ast }\varphi _{i}\delta _{i}=t_{i}^{2}\varphi $. This conclude the
proof that $\varphi $ is a rectangular extreme point of $\boldsymbol{K}$.

We are now ready to conclude the proof that $\boldsymbol{K}$ is the
rectangular convex hull of $\partial _{\rho }\boldsymbol{K}$. In view of the
rectangular bipolar theorem, it is enough to prove that if $n,m\in \mathbb{N}
$ and $z\in M_{n,m}(X)$ are such that $\left\Vert \varphi
^{(n,m)}(z)\right\Vert \leq 1$ for every $r\leq n$, $s\leq m$, and $\varphi
\in \partial _{\rho }K_{s,t}$, then $\left\Vert \psi ^{(n,m)}(z)\right\Vert
\leq 1$ for every $\psi \in K_{n,m}$. If $x\in M_{n,m}(X)$ then we let $%
\tilde{x}$ be the element%
\begin{equation*}
\begin{bmatrix}
I^{\oplus n} & x \\ 
x^{\ast } & I^{\oplus m}%
\end{bmatrix}%
\end{equation*}%
of $\tilde{X}$. Observe that if $\varphi \in M_{r,s}(X^{\prime })$ and $x\in
M_{n,m}(X)$, then%
\begin{equation*}
\tilde{\varphi}^{(n,m)}(x)=%
\begin{bmatrix}
I_{nr} & \varphi ^{(n,m)}(x) \\ 
\varphi ^{(n,m)}(x)^{\ast } & I_{ms}%
\end{bmatrix}%
\in M_{nr+ms}\left( \mathbb{C}\right) \text{.}
\end{equation*}%
If furthermore $\xi \in M_{r,n}\left( \mathbb{C}\right) $ and $\eta \in
M_{s,m}\left( \mathbb{C}\right) $ then%
\begin{equation*}
\left( \xi \odot \eta \right) ^{\ast }\tilde{\varphi}(x)\left( \xi \odot
\eta \right) =%
\begin{bmatrix}
I_{n}\otimes \xi ^{\ast }\xi & \left( \xi ^{\ast }\varphi \eta \right)
^{(n,m)}(x) \\ 
\left( \xi ^{\ast }\varphi \eta \right) ^{(n,m)}(x)^{\ast } & I_{m}\otimes
\eta ^{\ast }\eta%
\end{bmatrix}%
\text{.}
\end{equation*}%
Let now $\left( \xi \odot \eta \right) ^{\ast }\tilde{\varphi}\left( \xi
\odot \eta \right) $ be an extreme point of $\Delta _{n,m}$, where $\xi \in
M_{r,n}\left( \mathbb{C}\right) $ and $\eta \in M_{s,m}\left( \mathbb{C}%
\right) $ are right invertible and such that $\left\Vert \xi \right\Vert
_{2}=\left\Vert \eta \right\Vert _{2}=1$, and $\varphi \in \partial K_{r,s}$%
. By assumption we have that $\left\Vert \left( id_{M_{n,m}\left( \mathbb{C}%
\right) }\otimes \varphi \right) (z)\right\Vert \leq 1$. Thus by \cite[Lemma
3.1]{paulsen_completely_2002} we have%
\begin{equation*}
\begin{bmatrix}
I_{nr} & \varphi ^{(n,m)}(z) \\ 
\varphi ^{(n,m)}(z)^{\ast } & I_{ms}%
\end{bmatrix}%
\geq 0
\end{equation*}%
and hence%
\begin{equation*}
\left( \xi \odot \eta \right) ^{\ast }\tilde{\varphi}(z)\left( \xi \odot
\eta \right) =%
\begin{bmatrix}
I_{n}\otimes \xi ^{\ast }\xi & \left( \xi ^{\ast }\varphi \eta \right)
^{(n,m)}(z) \\ 
\left( \xi ^{\ast }\varphi \eta \right) ^{(n,m)}(z)^{\ast } & I_{m}\otimes
\eta ^{\ast }\eta%
\end{bmatrix}%
\geq 0\text{.}
\end{equation*}%
It follows from this and the classical Krein-Milman theorem that $\psi
(z)\geq 0$ for any $\psi \in \Delta _{n,m}$. Let us fix $\varphi \in K_{n,m}$%
. If $\xi =I_{n}$ and $\eta =I_{m}$ then $\left( \xi \odot \eta \right)
^{\ast }\tilde{\varphi}\left( \xi \odot \eta \right) \in \Delta _{n,m}$ and%
\begin{equation*}
\left( \xi \odot \eta \right) ^{\ast }\tilde{\varphi}(z)\left( \xi \odot
\eta \right) =%
\begin{bmatrix}
I_{n}\otimes \xi ^{\ast }\xi & \left( \xi ^{\ast }\varphi \eta \right)
^{(n,m)}(z) \\ 
\left( \xi ^{\ast }\varphi \eta \right) ^{(n,m)}(z)^{\ast } & I_{m}\otimes
\eta ^{\ast }\eta%
\end{bmatrix}%
=%
\begin{bmatrix}
I_{nr} & \left( \varphi ^{(n,m)}\right) (z) \\ 
\varphi ^{(n,m)}(z)^{\ast } & I_{ms}%
\end{bmatrix}%
\geq 0\text{.}
\end{equation*}%
This implies again by \cite[Lemma 3.1]{paulsen_completely_2002} that $%
\left\Vert \varphi ^{(n,m)}(z)\right\Vert \leq 1$. Since this is valid for
an arbitrary element of $K_{n,m}$, the proof is concluded.
\end{proof}

\begin{remark}
\label{Remark:krein-milman}The proof of Theorem \ref{Theorem:Krein-Milman}
shows something more precise: if $\boldsymbol{K}$ is a rectangular convex
set, then for every $n,m\in \mathbb{N}$, $K_{n,m}$ is equal to the closed
rectangular convex hull of $K_{r,s}$ for $r\leq n$ and $s\leq m$.
\end{remark}

The following is an immediate corollary of the rectangular Krein-Milman
theorem as formulated in Remark \ref{Remark:krein-milman}.

\begin{corollary}
\label{Corollary:krein-milman}Suppose that $X$ is an operator space, and $%
\boldsymbol{K}=\mathrm{C\mathrm{\mathrm{Ball}}}(X^{\prime })$ is the
corresponding compact rectangular matrix convex set. If $n,m\in \mathbb{N}$
and $x\in M_{n,m}(X)$, then $\left\Vert x\right\Vert $ is the supremum of $%
\left\Vert \varphi ^{(n,m)}(x)\right\Vert $ where $\varphi $ ranges among
all the rectangular extreme points of $K_{r,s}$ for $r\leq n$ and $s\leq m$.
\end{corollary}

\section{Boundary representations and the C*-envelope of a matrix-gauged
space\label{Section:gauge}}

\subsection{Selfadjoint operator spaces}

By a concrete\emph{\ selfadjoint} \emph{operator space }we mean a closed
selfadjoint subspace $X$ of $B(H)$. Any selfadjoint operator space is
endowed with a canonical involution, matrix norms, and matrix positive cones
inherited from $B(H)$.

An (abstract) matrix-ordered matrix-normed $\ast $-vector space---see \cite[%
Subsection 3.1]{russell_characterizations_2015}---is a vector space $V$
endowed with

\begin{itemize}
\item a conjugate-linear involution $v\mapsto v^{\ast }$,

\item a complete norm in $M_{n}\left( V\right) $ for every $n\in \mathbb{N}$,

\item a distinguished positive cone $M_{n}\left( V\right) _{+}\subset
M_{n}\left( V\right) $
\end{itemize}

such that, for every $n,k\in \mathbb{N}$, $x\in M_{n}(X)$, and $a,b\in
M_{n,k}\left( \mathbb{C}\right) $,

\begin{enumerate}
\item $M_{n}\left( V\right) _{+}$ is proper, i.e. $M_{n}\left( V\right)
_{+}\cap (-M_{n}\left( V\right) _{+})=\left\{ 0\right\} $,

\item $M_{n}\left( V\right) _{+}$ is closed in the topology induced by the
norm,

\item $\left\Vert a^{\ast }xb\right\Vert \leq \left\Vert a\right\Vert
\left\Vert x\right\Vert \left\Vert b\right\Vert $, and

\item when $a^{\ast }xa\in M_{k}\left( V\right)_+ $, when $x\in M_{n}\left( V\right)_+$.
\end{enumerate}

A matrix-ordered matrix-normed $\ast $-vector space $V$ is \emph{normal }if,
for every $n\in \mathbb{N}$ and $x,y,z\in M_{n}\left( V\right) $, $x\leq
y\leq z$ implies that $\left\Vert y\right\Vert \leq \max \left\{ \left\Vert
x\right\Vert ,\left\Vert z\right\Vert \right\} $. It is essentially proved
in \cite{werner_subspaces_2002}---see also \cite[Theorem 3.2]%
{russell_characterizations_2015} and \cite[Theorem 5.6]%
{russell_characterizations_2016}---that a matrix-ordered matrix-normed $\ast 
$-vector space $V$ is normal if and only if there exists a completely
positive completely isometric selfadjoint linear map $\phi :V\rightarrow
B(H) $, where $H$ is a Hilbert space and the space $B(H)$ of bounded linear
operators on $H$ is endowed with its canonical matrix-ordered matrix-normed $%
\ast $-vector space structure.

\subsection{Matrix-gauged spaces}

Suppose that $V$ is a real vector space. A \emph{gauge }over $V$ is a
subadditive and positively-homogeneous function $\nu :V\rightarrow \lbrack
0,+\infty )$. The conjugate gauge $\overline{\nu }$ is defined by $\overline{%
\nu }(x)=\nu (-x)$. The seminorm $\left\Vert x\right\Vert _{\nu }$
corresponding to a gauge is given by $\left\Vert x\right\Vert _{\nu }=\max
\left\{ \nu (x),\overline{\nu }(x)\right\} $. A gauge is proper if the
seminorm $\left\Vert \cdot \right\Vert _{\nu }$ is a norm. The positive cone
associated with a gauge $\nu $ is the set $V_{+,\nu }=\left\{ x\in V:%
\overline{\nu }(x)=0\right\} $.

The following notion is considered in \cite{russell_characterizations_2015}
under the name of $L^{\infty }$-matricially ordered vector space.

\begin{definition}
A \emph{matrix-gauged} space is a $\ast $-vector space $V$ endowed with a
sequence of proper gauges $\nu _{n}:M_{n}\left( V\right) _{sa}\rightarrow
\lbrack 0,+\infty )$ for $n\in \mathbb{N}$ with the property that, for every 
$n,k\in \mathbb{N}$, $x\in M_{n}\left( V\right) $, $y\in M_{k}\left(
V\right) $, and $a\in M_{n,k}\left( \mathbb{C}\right) $, one has that%
\begin{equation*}
\nu _{k}\left( a^{\ast }xa\right) \leq \left\Vert a\right\Vert ^{2}\nu
_{n}(x)
\end{equation*}%
and%
\begin{equation*}
\nu _{n+k}\left( x\oplus y\right) =\max \left\{ \nu _{n}(x),\nu
_{k}(y)\right\} \text{.}
\end{equation*}
\end{definition}

A linear map $\phi :V\rightarrow W$ between matrix-gauged spaces is
completely gauge-contractive if it is selfadjoint and $\nu \left( \phi
^{(n)}(x)\right) \leq \nu (x)$ for every $n\in \mathbb{N}$ and $x\in
M_{n}(X)_{sa}$, and completely gauge-isometric if it is selfadjoint and $\nu
\left( \phi ^{(n)}(x)\right) =\nu (x)$ for every $n\in \mathbb{N}$ and $x\in
M_{n}(X)_{sa}$; see \cite[Definition 3.11]{russell_characterizations_2015}.
Matrix-gauged spaces naturally form a category, where the morphisms are the
completely gauge-contractive maps, and isomorphism are completely
gauge-isometric surjective maps. In the following we will consider
matrix-gauged spaces as objects in this category. By \cite[Corollary 3.10]%
{russell_characterizations_2015}, every matrix-gauged space is completely
gauge-isometrically isomorphic to a concrete selfadjoint operator space.

Any matrix-gauged space $V$ has a canonical \emph{normal} matrix-ordered
matrix-normed $\ast $-vector space structure, obtained by considering the
gauge norms and the gauge cones associated with the given matrix-gauges.
Conversely, suppose that $V$ is a normal matrix-ordered matrix-normed $\ast $%
-vector space. Letting $\nu _{n}(x)$ be the distance of $x$ from $%
-M_{n}\left( V\right) _{+}$ for every $x\in M_{n}\left( V\right) _{sa}$
defines a canonical matrix-gauged structure on $X$. This matrix-gauge
structure induces the original matrix-order and matrix-norms on $V$ that one
started from; see \cite[Proposition 3.5]{russell_characterizations_2015}.
Furthermore a selfadjoint linear map $\phi :V\rightarrow B(H)$ is completely
positive and completely contractive if and only if it is completely
gauge-contractive with respect to these specific matrix-gauges. However,
there might be different matrix-gauges on $V$ that induce the same
matrix-order and matrix-norms on $V$.

Suppose now that $S$ is an operator system with order unit $1$. Then, in
particular, $S$ is a normal matrix-ordered matrix-normed $\ast $-vector
space. Furthermore, it admits a unique matrix-gauge structure compatible
with its matrix-order and matrix-norms. These matrix-gauges are defined by $%
\nu _{n}(x)=\inf \left\{ t>0:x\leq t1\right\} $ for a selfadjoint $x\in S$.
Uniqueness can be deduced from Arveson's extension theorem \cite[Theorem
1.6.1]{brown_c*-algebras_2008}, as proved in \cite[Theorem 6.9]%
{russell_characterizations_2016}. In the following we will regard an
operator system as a matrix-gauge space with such canonical matrix-gauges. A
unital selfadjoint linear map between operator systems is completely
positive if and only if it is completely gauge-contractive, and completely
isometric if and only if it is completely gauge-isometric.

It is proved in \cite[Subsection 3.3]{russell_characterizations_2015} that
any matrix-gauged space $W$ admits a completely gauge isometric embedding as
a subspace of codimension $1$ into an operator system $W^{\dag }$, called
the unitization of $W$, that satisfies the following universal property: any
completely gauge-contractive map from $W$ to an operator system $V$ admits a
unique extension to a unital completely positive map from $W^{\dag }$ to $V$%
. The unitization $W^{\dag }$ of $W$ is uniquely characterized by the above.
If $W$ is a normal matrix-ordered matrix-normed space, then we define the
unitization of $W$ to be the unitization of $W$ endowed with the canonical
matrix-gauges described above.

Suppose that $A$ is a (not necessarily unital) C*-algebra. Then $A$ is
endowed with canonical matrix-gauges, obtained by setting $\nu
_{n}(x)=\left\Vert x_{+}\right\Vert $ for a selfadjoint $x\in A$, where $%
x_{+}$ denotes the positive part of $x$. In the following we will consider a
C*-algebra as a matrix-gauged space with these canonical matrix-gauges. It
follows from the unitization construction that any matrix-gauged space
admits a completely gauge-isometric embedding into $B(H)$. The following
result can also be found in \cite[Proposition 2.2.1]{brown_c*-algebras_2008}
with a different proof.

\begin{lemma}
\label{Lemma:unitize-algebra}Suppose that $A\subset B(H)$ is a C*-algebra
such that the identity $1$ of $B(H)$ does not belong to $A$. Let $Y$ be an
operator system. Then for any completely positive completely contractive map 
$\phi :A\rightarrow Y$ there exists a unital completely positive map $\psi :%
\mathrm{span}\left\{ A,1\right\} \rightarrow Y$ extending $\phi $.
\end{lemma}

\begin{proof}
Let $Z:=\mathrm{span}\left\{ A,1\right\} \subset B(H)$ and assume that $Y
\subset B(L)$ for some Hilbert space $L$. We have to prove that if $%
z=x+\alpha 1\in M_{n}\left( Z\right) $ for $\alpha \in M_{n}\left( \mathbb{C}%
\right) $ is positive, then $\phi ^{(n)}(x)+\alpha 1\in M_{n}(Y)$ is
positive. Since $1\notin A$, we have that $\alpha $ is positive. Without
loss of generality, we can assume that $\alpha $ is invertible. After
replacing $z$ with $\alpha ^{-\frac{1}{2}}z\alpha ^{-\frac{1}{2}}$ we can
assume that $\alpha =1$. By Stinespring's theorem \cite[Theorem II.6.9.7]%
{blackadar_operator_2006}, there exist a Hilbert space $K$, a *-homomorphism 
$\pi :A\rightarrow B(K)$, and a linear map $v:L\rightarrow K$ such that $%
\left\Vert v\right\Vert =1$ and $\phi (x)=v^{\ast }\pi (x)v$ for every $x\in
A$. Observe that $\pi$ extends to a unital $*$-homomorphism from $Z$ into $%
B(L)$, which we still denote by $\pi$. Let $v^{(n)}:H^{(n)}\rightarrow
K^{(n)}$ be the map $v\oplus \cdots \oplus v$. Then we have that 
\begin{eqnarray*}
\phi ^{(n)}(x)+1 &=&v^{(n)\ast }\pi ^{(n)}(x)v^{(n)}+1 \\
&\geq &v^{(n)\ast }\pi ^{(n)}(x)v^{(n)}+v^{(n)\ast }\pi
^{(n)}(1)v^{(n)}=v^{(n)\ast }\pi ^{(n)}\left( x+1\right) v^{(n)}\geq 0\text{.%
}
\end{eqnarray*}%
This concludes the proof.
\end{proof}

It follows from the previous lemma that the unitization of a C*-algebra $A$
as a matrix-gauged space coincides with the unitization of $A$ as a
C*-algebra; see also \cite[Corollary 4.17]{werner_subspaces_2002}.
Furthermore, it follows from Lemma \ref{Lemma:unitize-algebra}, \cite[%
Theorem 6.9]{russell_characterizations_2016}, and Arveson's extension
theorem that a C*-algebra admits unique compatible matrix-gauges. One can
then deduce from \cite[Theorem 3.16]{russell_characterizations_2015} that a
linear map between C*-algebras or operator systems is completely
gauge-contractive if and only if it is completely positive contractive.

\subsection{The injective envelope of a matrix-gauged space}

We say that a matrix-gauged space is\emph{\ injective }if it is injective in
the category of matrix-gauged spaces and completely gauge-contractive maps.
Theorem 3.14 of \cite{russell_characterizations_2015} shows that $B(H)$ is
an injective matrix-gauged space when endowed with its canonical
matrix-gauges. It follows from this that the unitization functor $W\mapsto
W^{\dag }$ is an injective functor from the category of matrix-gauged spaces
and gauge-contractive maps to the category of operator systems and unital
completely positive maps.

Our goal now is to show that any injective matrix-gauged space is
(completely gauge-isometrically isomorphic to) a unital C*-algebra. This is
a generalization of a theorem of Choi and Effros see \cite[Theorem 15.2]%
{paulsen_completely_2002}.

\begin{proposition}
\label{Proposition:CE} Let $X$ be an injective matrix-gauged space. Then $X$
is completely gauge-isometrically isomorphic to a unital C*-algebra.
\end{proposition}

\begin{proof}
We may assume that $X\subset B(H)$ is concretely represented as a
selfadjoint operator space. Since $X$ is injective, there exists a
gauge-contractive and hence completely contractive and completely positive
projection $\Phi :B(H)\rightarrow X$. We define the Choi-Effros product on $%
X $ by 
\begin{equation*}
x\cdot _{\Phi }y=\Phi (xy).
\end{equation*}%
As in the proof of \cite[Theorem 15.2]{paulsen_completely_2002}, one shows
that 
\begin{equation*}
\Phi (\Phi (a)x)=\Phi (ax)\quad \text{ and }\quad \Phi (x\Phi (a))=\Phi (xa)
\end{equation*}%
holds for all $x\in X$ and all $a\in B(H)$. Indeed, the proof only requires
the Schwarz inequality for unital completely positive maps, which remains
valid for completely positive completely contractive maps. In particular, we
see that $e:=\Phi (I_{H})\in X$ is a unit for the Choi-Effros product.
Moreover, the proof of \cite[Theorem 15.2]{paulsen_completely_2002} shows
that $(X,\cdot _{\Phi })$, endowed with the norm and involution of $B(H)$,
is a C*-algebra.

It is clear from the above that the identity map from $X$ onto $(X,\cdot
_{\Phi })$ is an isometry. To see that it is an order isomorphism, suppose
that $x\in X$ is positive with respect to the order on $B(H)$. Then $%
||cI_{H}-x||\leq c$ for all $c\geq ||x||$, so since $\Phi $ is contractive $%
||ce-x||\leq c$ for all $c\geq ||x||$, thus $x$ is a positive element of the
C*-algebra $(X,\cdot _{\Phi })$. Conversely, if $x$ is positive in the
C*-algebra $(X,\cdot _{\Phi })$, then there exists $y\in X$ such that $%
x=y^{\ast }\cdot _{\Phi }y$, hence $x=\Phi (y^{\ast }y)$ is positive in $%
B(H) $. Moreover, the argument at the end of the proof of \cite[Theorem 15.2]%
{paulsen_completely_2002} shows that $M_{n}(X)$, endowed with the
Choi-Effros product $\cdot _{\Phi ^{(n)}}$, is the C*-tensor product of $%
(X,\cdot _{\Phi })$ with $M_{n}(\mathbb{C})$. By the above, the identity map
is an isometry and an order isomorphism between $M_{n}(X)\subset M_{n}(B(H))$
and $(M_{n}(X),\cdot _{\Phi })$. Therefore, the identity map from $X$ onto $%
(X,\cdot _{\Phi ^{(n)}})$ is a selfadjoint complete isometry and complete
order isomorphism.

To see that the identity map from $X$ onto $(X,\cdot _{\Phi })$ is in fact a
complete gauge isometry, observe that $\Phi $ is a unital completely
positive map from $B(H)$ onto $(X,\cdot _{\Phi })$ by the preceding
paragraph, so it is completely gauge contractive. Conversely, if $x\in
M_{n}(X)$ is self-adjoint and satisfies $||x_{+}||\leq 1$, where the
positive part is taken in the C*-algebra $(M_{n}(X),\cdot _{\Phi })$, then $%
x\leq e^{(n)}$ in $(M_{n}(X),\cdot _{\Phi })$, hence $x\leq I_{\mathbb{C}%
^{n}}\otimes I_{H}$ in $M_{n}(B(H))$ by the preceding paragraph, so that the
identity map from $(X,\cdot _{\Phi })$ to $X$ is completely gauge
contractive as well.
\end{proof}

In particular, we see that every injective matrix-gauged space is
(completely gauge-isometrically isomorphic to) an injective operator system.
Conversely, since the unitization functor is injective, every operator
system that is injective in the category of operator systems and unital
completely positive maps is also injective as a matrix-gauged space, when
endowed with the unique compatible matrix-gauge structure.

The usual proof of the existence of the injective envelope of an operator
system yields the existence of a gauge analog of Hamana's injective envelope
of operator spaces. Let us say that a gauge-extension of a matrix-gauged
space $X$ is a pair $\left( Y,i\right) $ where $Y$ is a matrix-gauged space
and $i:X\rightarrow Y$ is a completely gauge-isometric map. As in the case
of operator systems, we say that such a gauge-extension is:

\begin{enumerate}
\item \emph{rigid} if the identity map of $Y$ is the unique
gauge-contractive map $\phi :Y\rightarrow Y$ such that $\phi \circ i=i$;

\item \emph{essential }if whenever $u:Y\rightarrow Z$ is a gauge-contractive
map to a matrix-gauged space $Z$ such that $u\circ i$ is a completely
gauge-isometric, then $u$ is a completely gauge-isometric;

\item an \emph{injective envelope }if $Y$ is injective, and there is no
proper injective subspace of $Y$ that contains $X$.
\end{enumerate}

The same proof as \cite[Lemma 4.2.4]{blecher_operator_2004} shows that if $X$
is a matrix-gauged space, and $\left( Y,i\right) $ is a gauge-extension of $%
X $ such that $Y$ is injective, then the following assertions are
equivalent: 1) $\left( Y,i\right) $ is an injective envelope of $X$; 2) $%
\left( Y,i\right) $ is essential; 3) $(Y,i)$ is rigid. To this purpose one
can consider the gauge analog of the notion of projections and seminorms
from \cite[Subsection 4.2.1]{blecher_operator_2004}.

Suppose that $W$ is a matrix-gauged space, and $X$ is a selfadjoint subspace
of $W$. A completely gauge-contractive $X$-projection on $W$ is an
idempotent completely gauge-contractive map $u:W\rightarrow W$ that
restricts to the identity on $X$. A gauge $X$-seminorm on $W$ is a seminorm
of the form $p(x)=\left\Vert u(x)\right\Vert $ for some completely
gauge-contractive $X$-projection $u$ on $W$. One can define an order on
completely gauge-contractive $X$-projections by $u\leq v$ if and only if $%
u\circ v=v\circ u=u$, while gauge $X$-seminorms are ordered by pointwise
comparison. The same proof as \cite[Lemma 4.2.2]{blecher_operator_2004}
shows that any gauge $X$-seminorm majorizes a minimal gauge $X$-seminorm,
and if $p$ is a minimal gauge $X$-seminorm and $u:W\rightarrow W$ is a
completely gauge-contractive map that restricts to the identity on $X$, then 
$u$ is a minimal gauge $X$-projection. To this purpose, it is enough to
observe that the set of completely gauge-contractive selfadjoint maps from $%
W $ to $B(H)$ is closed in the weak* topology of the space of $\mathrm{CB}%
\left( W,B(H)\right) $ of completely bounded maps from $W$ to $B(H)$. Indeed 
$\phi :W\rightarrow B(H)$ is completely gauge-contractive if and only if it
is selfadjoint and $\left\langle \phi ^{(n)}(x)\xi ,\xi \right\rangle \leq
\nu (x)$ for every $n\in \mathbb{N}$, $\xi \in H^{\oplus n}$, and $x\in
M_{n}(W)_{sa}$.

The proof of \cite[Lemma 4.2.4]{blecher_operator_2004} can now be easily
adapted to prove the claim above, by replacing $X$-projections with gauge $X$%
-projections and $X$-seminorms with gauge $X$-seminorms. Similarly the same
proof as \cite[Theorem 4.2.6]{blecher_operator_2004} shows that if a
matrix-gauged space $X$ is contained in an injective matrix-gauged space $W$%
, then there exists an injective envelope $X\subset Z\subset W$. Furthermore
the injective envelope of $X$ is essentially unique. We denote by $I(X)$ the
injective envelope of a matrix-gauged space $X$, and we identify $X$ with a
selfadjoint subspace of $I(X)$. It is clear that, when $X$ is an operator
system endowed with its canonical matrix-gauges, the injective envelope of $%
X $ as a matrix-gauged space coincides with the injective envelope of $X$ as
an operator system (endowed with the canonical matrix-gauges). Furthermore,
it is a consequence of Proposition \ref{Proposition:CE} that the unitization
of a matrix-gauged space $X$ is $\mathrm{span}\left\{ X,1\right\} \subset
I(X)$, where $1$ denotes the identity of the unital C*-algebra $I(X)$.

\subsection{Boundary representations}

Most fundamental notions in dilation theory admit straightforward versions
in the setting of matrix-gauged spaces. Suppose that $X$ is a matrix-gauged
space. An \emph{operator} \emph{state }on $X$ is a completely
gauge-contractive map $\phi :X\rightarrow B(H)$. We say that an operator
state $\psi :X\rightarrow B(\widetilde{H})$ is a \emph{dilation }of $\phi $
if there exists a linear isometry $v:H\rightarrow \widetilde{H}$ such that $%
v^{\ast }\psi (x)v=\phi (x)$ for every $x\in X$. It follows from
Stinespring's dilation theorem \cite[Theorem II.6.9.7]%
{blackadar_operator_2006} that if $A$ is a C*-algebra, then an operator
state on $A$ admits a dilation which is a *-homomorphism. A dilation $\psi $
of an operator state $\phi :x\mapsto v^{\ast }\psi (x)v$ on $X$ is \emph{%
trivial }if $\psi (x)=v v^* \psi (x)v v^{\ast } +(1-v v^*)\psi (x)\left( 1- v v^{\ast
}\right) $. We say that $\phi $ is \emph{maximal }if it has no nontrivial
dilation. As in the case of operator systems, one can prove that an operator
state $\phi :X\rightarrow B(H)$ is maximal if and only if for any dilation $%
\psi $ of $\phi $ one has that $\left\Vert \psi (x)\xi \right\Vert
=\left\Vert \phi (x)\xi \right\Vert $ for every $x\in X$, and $\xi \in H$.

Suppose that $X$ is a selfadjoint subspace of a C*-algebra $A$ such that $A$
is generated as a C*-algebra by $X$. An operator state $\phi $ on $X$ has the%
\emph{\ unique extension property} if any completely positive contractive
map $\widetilde{\phi }:A\rightarrow B(H)$ whose restriction to $X$ coincides
with $\phi $ is automatically a *-homomorphism. The same argument as in the
operator systems setting shows that an operator state is maximal if and only
if it has the unique extension property; see \cite%
{arveson_noncommutative_2008}.

\begin{definition}
\label{Definition:bound-rep ordered}A \emph{boundary representation }for a
matrix-gauged space $X\subset B(H)$ is an operator state $\phi :X\rightarrow
B(H)$ with the property that any completely positive contractive map $\psi
:C^{\ast }(X)\rightarrow B(H)$ extending $X$ is an irreducible
representation of $C^{\ast }(X)$.
\end{definition}

In the following we will identify a boundary representation of $X$ with its
unique extension to an irreducible representation of $C^{\ast }(X)$. It
follows from the remarks above that the notion of boundary representation
does not depend on the concrete realization of $X$ as a selfadjoint space of
operators. In the following we will assume that $A$ is a C*-algebra, and $%
X\subset A\ $is a selfadjoint subspace that generates $A$ as a C*-algebra.
We regard $X$ as a matrix-gauged space endowed with the matrix-gauges
induced by $A$.

\begin{proposition}
\label{Proposition:unitize-boundary}Suppose that $\phi :X\rightarrow B(H)$
is an operator state of $X$, and $\phi ^{\dagger }:X^{\dagger }\rightarrow
B(H)$ is its canonical unital completely positive extension to the
unitization of $X$. If $\phi ^{\dagger }$ is a boundary representation for $%
X^\dagger$, then $\phi $ is a boundary representation for $X$.
\end{proposition}

\begin{proof}
Let $\Phi :A\rightarrow B(H)$ be a completely positive contractive map
extending $\phi $. Extend $\Phi $ to a unital completely positive $\Phi
^{\dag }:A^{\dagger }\rightarrow B(H)$. Then since by assumption $\phi
^{\dagger }$ is a boundary representation, we conclude that $\Phi ^{\dag }$
is an irreducible representation for $A^{\dagger }$. Therefore $\Phi |_{A}$
is an irreducible representation of $A$. This concludes the proof.
\end{proof}

The following result is then a consequence of Proposition \ref%
{Proposition:unitize-boundary} and \cite[Theorem 3.1]{davidson_choquet_2013}.

\begin{theorem}
\label{Theorem:enough-boundary-ordered}Suppose that $X$ is a matrix-gauged
space. Then the matrix-gauges of $X$ are completely determined by the
boundary representations of $X$. Precisely, if $x\in M_{n}(X)$, then $\nu
_{n}(x)$ is the supremum of $\left\Vert \phi ^{(n)}(x)_{+}\right\Vert $
where $\phi $ ranges among all the boundary representations of $X$.
\end{theorem}

Suppose that $X$ is a matrix-gauged space, and $\phi :X\rightarrow B(H)$ is
an operator state. An\emph{\ operator convex combination} is an expression $%
\phi =\alpha _{1}^{\ast }\phi _{1}\alpha _{1}+\cdots +\alpha _{n}^{\ast
}\phi _{n}\alpha _{n}$, where $\alpha _{i}:H\rightarrow H_{i}$ are linear
maps, and $\phi _{i}:X\rightarrow B(H_{i})$ are operator states for $%
i=1,2,\ldots ,\ell $. Such a rectangular convex combination is \emph{proper }%
if the $\alpha _{i}$'s are right invertible and $\alpha _{1}^{\ast }\alpha
_{1}+\cdots +\alpha _{n}^{\ast }\alpha _{n}=1$ and \emph{trivial }if $\alpha
_{i}^{\ast }\alpha _{i}=\lambda _{i}1$ and $\alpha _{i}^{\ast }\phi
_{i}\alpha _{i}=\lambda _{i}\phi $ for some $\lambda _{i}\in \left[ 0,1%
\right] $.

\begin{definition}
An operator state $\phi :X\rightarrow B(H)$ is an\emph{\ operator extreme
point} if for any proper operator convex combination $\phi =\alpha
_{1}^{\ast }\phi _{1}\alpha _{1}+\cdots +\alpha _{n}^{\ast }\phi _{n}\alpha
_{n}$ is trivial.
\end{definition}

Observe that the map $\phi \mapsto \phi ^{\dagger }$ establishes a 1:1
correspondence between operator states on $X$ and operator states on the
operator system $X^{\dagger }$. Furthermore this correspondence is \emph{%
operator affine }in the sense that it preserves operator convex
combinations. The following proposition is then an immediate consequence of
this observation and (the proof of) \cite[Theorem B]{farenick_extremal_2000}.

\begin{proposition}
\label{Proposition:characterize-extreme}Suppose that $\phi :X\rightarrow
B(H) $ is an operator state, and let $\phi ^{\dagger }:X^{\dagger
}\rightarrow B(H)$ be its unital extension to the unitization of $X$. The
following assertions are equivalent:

\begin{enumerate}
\item $\phi $ is a pure element in the cone of completely gauge-contractive
maps from $X$ to $B(H)$;

\item $\phi $ is an operator extreme point;

\item $\phi ^{\dagger }$ is an operator extreme point.
\end{enumerate}
\end{proposition}

The following corollary is an immediate consequence of Proposition \ref%
{Proposition:characterize-extreme},\ Proposition \ref%
{Proposition:unitize-boundary}, and \cite{davidson_choquet_2013}.

\begin{corollary}
Suppose that $X$ is a matrix-gauged space, and $\phi $ is an operator state
on $X$. If $\phi $ is operator extreme, then $\phi $ admits a dilation to a
boundary representation of $X$.
\end{corollary}

\subsection{The C*-envelope of a matrix-gauged space \label%
{Subsection:C*-envelope}}

Suppose that $X$ is a matrix-gauged space. A pair $\left( A,i\right) $ is a
C*-cover if $A$ is a C*-algebra and $i:X\rightarrow A$ is a completely
gauge-isometric map whose range generates $A$ as a C*-algebra.

\begin{definition}
A C*-\emph{envelope} $\left( C_{e}^{\ast }(X),i\right) $ of $X$ is a C*-cover%
\emph{\ }of $X$\emph{\ }if it has the following universal property: for any
C*-cover $\left( B,j\right) $ of $X$, there exists a *-homomorphism $\theta
:B\rightarrow C_{e}^{\ast }(X)$ such that $\theta \circ j=i$.
\end{definition}

It is clear that the C*-envelope of a matrix-gauged space, if it exists, it
is essentially unique. We will prove below that any matrix-gauged space has
a C*-envelope. The proof is essentially the same as the one for the
existence of the C*-envelope of an operator system.

Suppose that $X$ is an matrix-gauged space. Let $X\subset I(X)$ be the
injective envelope of $X$. By Proposition \ref{Proposition:CE}, $I(X)$ is a
unital C*-algebra. Let $A $ be the C*-subalgebra of $I(X)$ generated by $X$.
As in the proof of \cite[Theorem 15.16]{paulsen_completely_2002}, one sees
that $\left( i,A\right) $ where $i:X\rightarrow A$ is the inclusion map, is
the C*-envelope of $X$.

It follows from the construction that the C*-envelope (as defined above) of
an operator system regarded as matrix-gauged space with its unique
compatible matrix-gauges coincides with the usual notion of C*-envelope of
an operator system.

Alternatively, one can construct the C*-envelope of a matrix-gauged space
using boundary representations, as for the C*-envelope of an operator
system. Indeed let $X$ be a matrix-gauged space. Define $\iota
_{e}:X\rightarrow B(H)$ to be the direct sum of all the boundary
representations for $X$, and then let $A$ be the C*-subalgebra of $B(H)$
generated by the image of $\iota _{e}$. It follows from the unique extension
property of boundary representations that $\left( \iota _{e},A\right) $ is
indeed the C*-envelope of $X$. In particular, this construction shows that,
for any C*-algebra $B$, $C_{e}^{\ast }\left( B\right) =B$.

As in the case of operator systems and operator spaces, one can define a
maximal or universal C*-algebra that contains a given ordered operator space
as a generating subset. Explicitly, the \emph{maximal }C*-algebra $C_{\max
}^{\ast }(X)$ of an ordered operator space is a C*-cover $\left( i,A\right) $
of $X$ that has the following universal property: given any other completely
gauge-contractive map $f:X\rightarrow B$, where $B$ is a C*-algebra, there
exists a *-homomorphism $\theta :A\rightarrow B$ such that $\theta \circ i=f$%
. In order to see that such a maximal C*-algebra exists, one can consider
the collection $\mathcal{F}$ of all completely gauge-contractive maps from $%
X $ to $M_{n}\left( \mathbb{C}\right) $ for $n\in \mathbb{N}$. Then let $i$
be the direct sum of the elements $s:X\rightarrow M_{n_{s}}\left( \mathbb{C}%
\right) $ of $\mathcal{F}$, and then $A$ to be the C*-subalgebra of $%
\bigoplus\nolimits_{s\in \mathcal{F}}^{\infty }M_{n_{s}}\left( \mathbb{C}%
\right) $ generated by the image of $i$. The same proof as \cite[Proposition
8]{kirchberg_c*-algebras_1998} shows that such a C*-cover satisfies the
required universal property.

\subsection{Selfadjoint ordered operator spaces and compact matrix convex
sets}

We want to conclude by observing that selfadjoint operator spaces are in
canonical 1:1 correspondence with compact matrix convex sets with a
distinguished extreme point.

Suppose that $\boldsymbol{K}=\left( K_{n}\right) $ is a compact matrix
convex set, and $e\in K_{1}$ is a matrix extreme point. Define $A_{0}(%
\boldsymbol{K},e)$ to be the set of continuous matrix-affine functions from $%
\boldsymbol{K}$ to $\left( M_{n}(\mathbb{C})\right) _{n\in \mathbb{N}}$ that
vanish at $e$. Then $A_{0}(\boldsymbol{K},e)$ is a selfadjoint subspace of
codimension $1$ of the operator system $A(\boldsymbol{K})$.

Conversely suppose that $X\subset B(H)$ is a selfadjoint operator space.
Consider $X$ as a normal matrix-ordered and matrix-normed space with respect
to the induced matrix-cones and matrix-norms, and let $X^{\dag }$ be the
unitization of $X$. Let, for $n\in \mathbb{N}$, $K_{n}$ be the space of
completely positive completely contractive selfadjoint maps from $X$ to $%
M_{n}\left( \mathbb{C}\right) $, endowed with the topology of pointwise
convergence. Observe that $\boldsymbol{K}=\left( K_{n}\right) $ is a compact
matrix-convex set, and $K_{n}$ can be identified with the space of unital
completely positive maps from $X^{\dag }$ to $M_{n}\left( \mathbb{C}\right) $%
. Let $e\in K_{1}$ be the zero functional on $X$. We have a canonical unital
complete order isomorphism $X^{\dag }\cong A(\boldsymbol{K})$. Under this
isomorphism $X$ is mapped into $A_{0}(\boldsymbol{K},e)$. Since $X$ has
codimension $1$ in $X^{\dag }$, such an isomorphism in fact maps $X$ onto $%
A_{0}(\boldsymbol{K},e)$.

The above construction shows that one can identify the unitization of $A_{0}(%
\boldsymbol{K},e)$ with the operator system $A(\boldsymbol{K})$. Furthermore
it is easy to see that the correspondence $(K,e)\mapsto A_{0}(\boldsymbol{K}%
,e)$ is a contravariant equivalence of categories from the category of
compact matrix convex sets $\boldsymbol{K}$ with a distinguished matrix
extreme point $e\in K_{1}$, where morphisms are continuous matrix-affine
maps that preserve the distinguished point, to the category of selfadjoint
operator spaces and completely positive completely contractive selfadjoint
maps.

The commutative analog of the argument above establishes a correspondence
between compact convex sets with a distinguished extreme point and
selfadjoint operator spaces that can be represented inside an abelian
C*-algebra (selfadjoint function spaces).

\bibliographystyle{amsplain}
\bibliography{bibliography}

\end{document}